\documentclass[10pt]{article}
\usepackage{amsmath, amsthm,amsfonts,amssymb,url,color,epsfig,graphics}
\usepackage[english]{babel}
\usepackage{float}

\usepackage{enumitem}
\usepackage{graphicx}
\usepackage[hang,small]{caption}
\usepackage[active]{srcltx}

\newtheorem{thm}{Theorem}[section]
\newtheorem{Prop}[thm]{Proposition}
\newtheorem{lem}[thm]{Lemma}

\newtheorem{defi}[thm]{Definition}
 \newtheorem{Thm}{Theorem}[section]
 \newtheorem{Rmk}[thm]{Remark}
 \newtheorem{Lem}[thm]{Lemma}
  \newtheorem{Def}[thm]{Definition}
   
\def\bC {\mathbb{C}}

 \def\N {\mathbb{N}}
\def\R {\mathbb{R}}

\def\K {\mathbb{K}}
\def\V {\mathbb{V}}
 \def\bP {\mathbb{P}}

\def\cA {\mathcal{A}}
\def\cB {\mathcal{B}}

\def\cL {\mathcal{L}}

\def\cW {\mathcal{W}}

\newcommand{\Winc}{\cW^0_{\text{inc}}}
\newcommand{\WBLII}{\cW_{BL,\eps^2}}
\newcommand{\WBLIII}{\cW_{BL,\eps^3}}
\newcommand{\WMF}{\cW_{MF}^1}
\newcommand{\WSH}{\cW_{II}^1}
\newcommand{\cU}{\mathcal U}

\newcommand{\cV}{\mathcal V}

\def\b {{\beta}}

\def\eps {{\varepsilon}}

\def\om {{\omega}}

\def\d {{\partial}}

\newcommand{\ba}{\begin{aligned}}
\newcommand{\ea}{\end{aligned}}
\newcommand{\be}{\begin{equation}}
\newcommand{\ee}{\end{equation}}

\newcommand{\C}{\mathbb C}

\numberwithin{equation}{section}

\begin{document}


\title{Near-critical reflection of internal waves}
\author{Roberta Bianchini 
 \footnote{Sorbonne Universit\'e, LJLL, UMR CNRS-UPMC 7598, 4 place Jussieu, 75252-Paris Cedex 05
(France) {\color{blue} {\tt bianchini@ljll.math.upmc.fr}}}~~ \& \\
Anne-Laure Dalibard
 \footnote{Sorbonne Universit\'e, LJLL, UMR CNRS-UPMC 7598, 4 place Jussieu, Boite courrier 187, 75252-Paris Cedex 05
(France) {\color{blue} {\tt dalibard@ljll.math.upmc.fr}}}~~ \& \\
Laure Saint-Raymond
 \footnote{Ecole Normale Sup\'erieure de Lyon,  UMPA,  UMR CNRS-ENSL 5669,  46, all\'ee d'Italie,  69364-Lyon Cedex 07
(France) {\color{blue} {\tt laure.saint-raymond@ens-lyon.fr}}}
}

\maketitle 
\begin{abstract} 
Internal waves describe the (linear) response of an incompressible stably stratified fluid to small perturbations.
The inclination  of their group velocity with respect to the vertical is completely determined by their frequency.
Therefore the reflection on a sloping boundary cannot follow Descartes' laws, and it is expected to be singular if the slope
has the same inclination as the group velocity.
In this paper, we prove that in this critical geometry the weakly viscous and weakly nonlinear wave equations have actually a solution which is 
well approximated by the sum of the incident wave packet, a reflected second harmonic and some boundary layer terms.
This result confirms the prediction by Dauxois and Young, and provides precise estimates on the time of validity of this approximation.
\end{abstract}


\section{Main results}

 Internal waves are of utmost importance in oceanic flows. They describe small departures from equilibrium in an incompressible fluid under the combined effect of density stratification and  gravity.
These waves are very well described when the effect of boundaries is neglected, assuming for instance that the fluid is contained in a parallelepipedic box with zero flux condition at the boundary (or equivalently in a periodic box). 

\bigskip
Assume that at equilibrium the stratification is given by the stable profile $\bar \rho (x) = \bar \rho(x_3)$ with $\bar \rho '(x_3) <0$. 
Note that, in most physical systems, the variations of $\bar \rho$ are very small compared to its average $\rho_0$, and count only for the buoyancy effect (not for inertia). 
Small perturbations of this equilibrium will create both a (zero divergence) velocity field $\delta v$, and a fluctuation of the density $ \rho = \bar \rho +\delta \dfrac{\rho_0}{g}  b$ of order $\delta \ll 1$, where  $b$ is the buoyancy.
The dynamics of the system is then governed by the {\bf Boussinesq equations}~:
\begin{equation}
\label{boussinesq}
\begin{aligned}
\partial_t v  +  \delta (v \cdot \nabla )  v + \nabla P + b e_3 ={\nu}  \Delta v, \\
\partial_t b +\delta ( v\cdot \nabla ) b  +  { g  \over \rho_0} v_3 \bar \rho'   =\kappa  \Delta b, \\
\nabla\cdot v =0,
\end{aligned}
\end{equation}
see \cite{Michel, Charve, Alazard,DH} for further details on the derivation of the system.
Note that the viscous term  $\kappa  \Delta b$ in the buoyancy equation  comes from the thermal dissipation (combined with the Boussinesq approximation connecting the density and the temperature, see \cite{DY}).
One often considers in addition that the stratification is locally affine so that $\bar \rho'$ is a constant, see again \cite{DY}. For the sake of simplicity, we will focus on this case.
Note however that this assumption is not consistent  with the approximation $\bar \rho \sim \rho_0$ in a large domain, especially in the whole space or in a half space.

\bigskip
Keeping only the leading order terms (the linear inviscid approximation) in (\ref{boussinesq}),  and taking the Fourier transform of this linear system with constant coefficients,
we obtain the {\bf linear inviscid wave equation}
\begin{equation}
\label{wave}
\d _t \begin{pmatrix} \hat v_{p,3} \\ \hat b_p \end{pmatrix} + \begin{pmatrix} 0&   1- {p_3^2\over |p|^2}\\ {g \bar \rho' \over \rho_0}  & 0 \end{pmatrix} \begin{pmatrix}  \hat v_{p,3} \\ \hat b_p\end{pmatrix} = 0\,.
\end{equation}
The solution in $\mathbb{R}^3$ can be expressed as a sum of plane waves with dispersion relation
$$\omega = \pm N {|p_h| \over |p|}\,,$$
where $N= ( - g \bar \rho' /\rho_0)^{1/2}$ denotes the Brunt-V\"ais\"al\"a frequency (see \cite{Michel} for further details).
By analogy with geometric optics for electromagnetic or acoustic waves (see for instance \cite{JMR1, JMR2, Metivier2}), one can use this dispersion relation to study the  ``propagation" of internal waves (which actually makes sense only for wave packets and not for single plane waves, see \cite{NLM}). The crucial difference here is that the frequency $\omega$  prescribes the direction of the propagation 
instead of the modulus of the wavelength~:  in 2D, on the energy level $\omega $, the wavenumber $p$ satisfies indeed
 $${|p_1|\over |p|} =\frac1N  \omega$$
 and the group velocity $\nabla_p \omega$,
 which is orthogonal to $p$, makes an angle $\beta=\pm  \arcsin (\omega/N)$  with respect to the horizontal (see also \cite{DVSR} for a mathematical description of attractors for waves with homogeneous dispersion relation of degree 0).

\bigskip
What we would like to understand in this paper is how these waves behave {\bf in presence of a sloping boundary}. This is a major challenge from the physical point of view. Sandstrom \cite{S}  proposed indeed in 1966 the oceanic internal wave field as a
possible source of the energy which is needed to activate strong mixing  near sloping
boundaries.

\subsection{Physical predictions for the near-critical reflection}

$\bullet$ {\bf The reflection of  inviscid internal waves} off a uniformly sloping bottom was first investigated by Philipps \cite {Ph} in 1966.
Because the
wave frequency is related to the direction of propagation by $\omega = N\sin \beta$,  preservation
of $\omega$  in the reflection implies preservation of the angle $\beta$ (which is  the angle between the group velocity and the horizontal, or equivalently between the phase velocity and the vertical). A simple geometric picture (see Figure \ref{fig: focusing})
shows then that this will generate a focusing mechanism.
\begin{figure} [h] 
\centering
\caption{\small  
Focusing of waves reflected by a sloping bottom.}
\label{fig: focusing}
\includegraphics[scale=0.3]{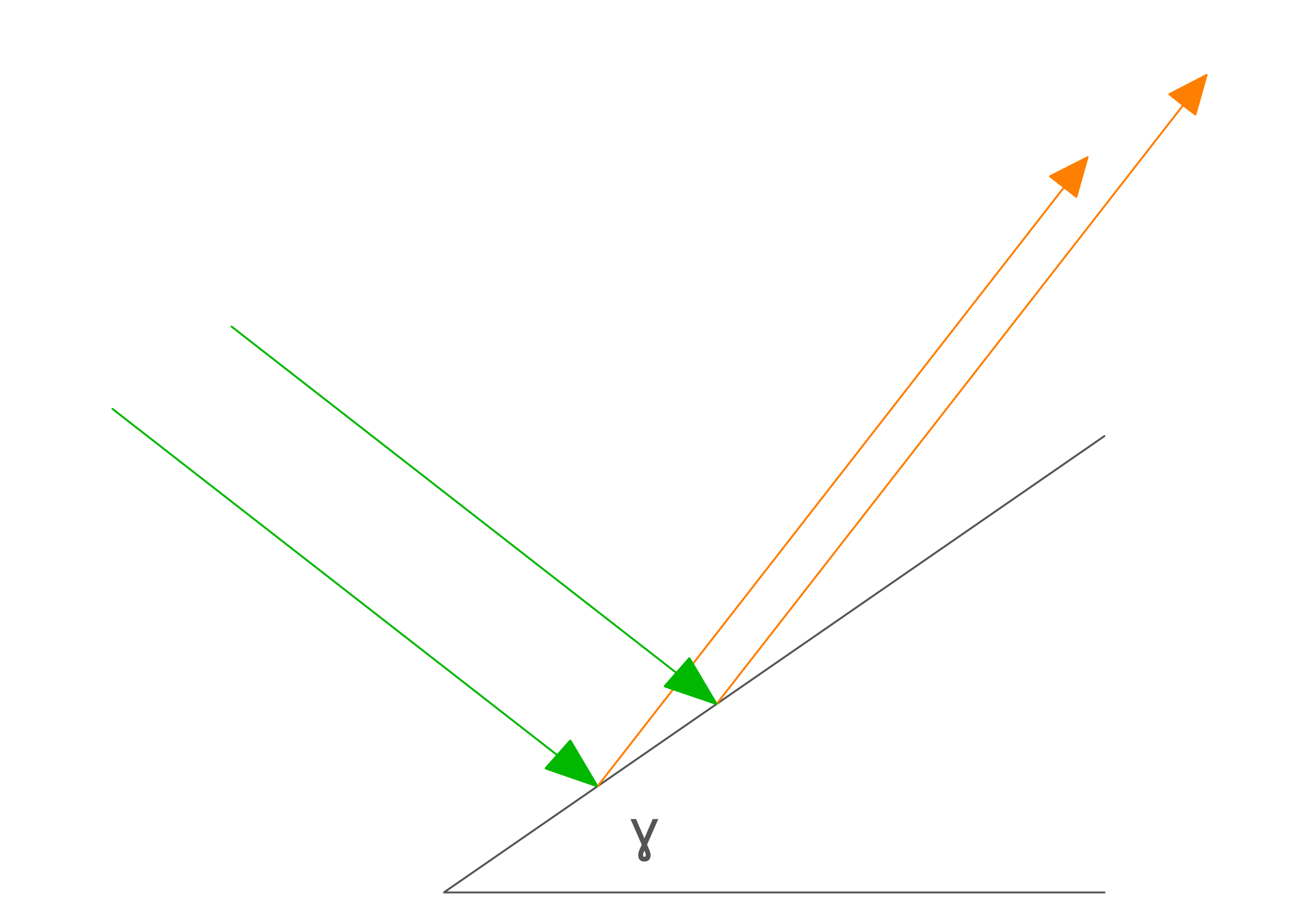}
\end{figure}
More precisely, to determine the reflection laws, we look at the solutions of the linear inviscid wave equation (\ref{wave}) in the half space delimited by the slope 
    $$ x_1 \sin\gamma - x_3 \cos\gamma = 0\,.$$ 
    We seek these solutions in the form of an incident wave  propagating with an angle $\beta$ with respect to the horizontal (recall that the direction of propagation is orthogonal to the wavenumber) plus a reflected wave.
    In order that the zero flux condition (which is the only admissible boundary condition in the inviscid regime) is satisfied on the slope, we then obtain necessary conditions on the wavenumber of the reflected wave,  as well as some polarization conditions to determine the amplitude of the reflected wave. To simplify our study, we work in a 2D setting from now on.

    An appropriate system of coordinates to express these conditions is $(x,y)$ where $x$ is the abscissa along the boundary and $y$ the distance to the boundary (see Figure \ref{fig: coordinates}).
\begin{figure} [h] 
\centering
\caption{\small  
Reference system of coordinates.}
\label{fig: coordinates}
\includegraphics[scale=0.3]{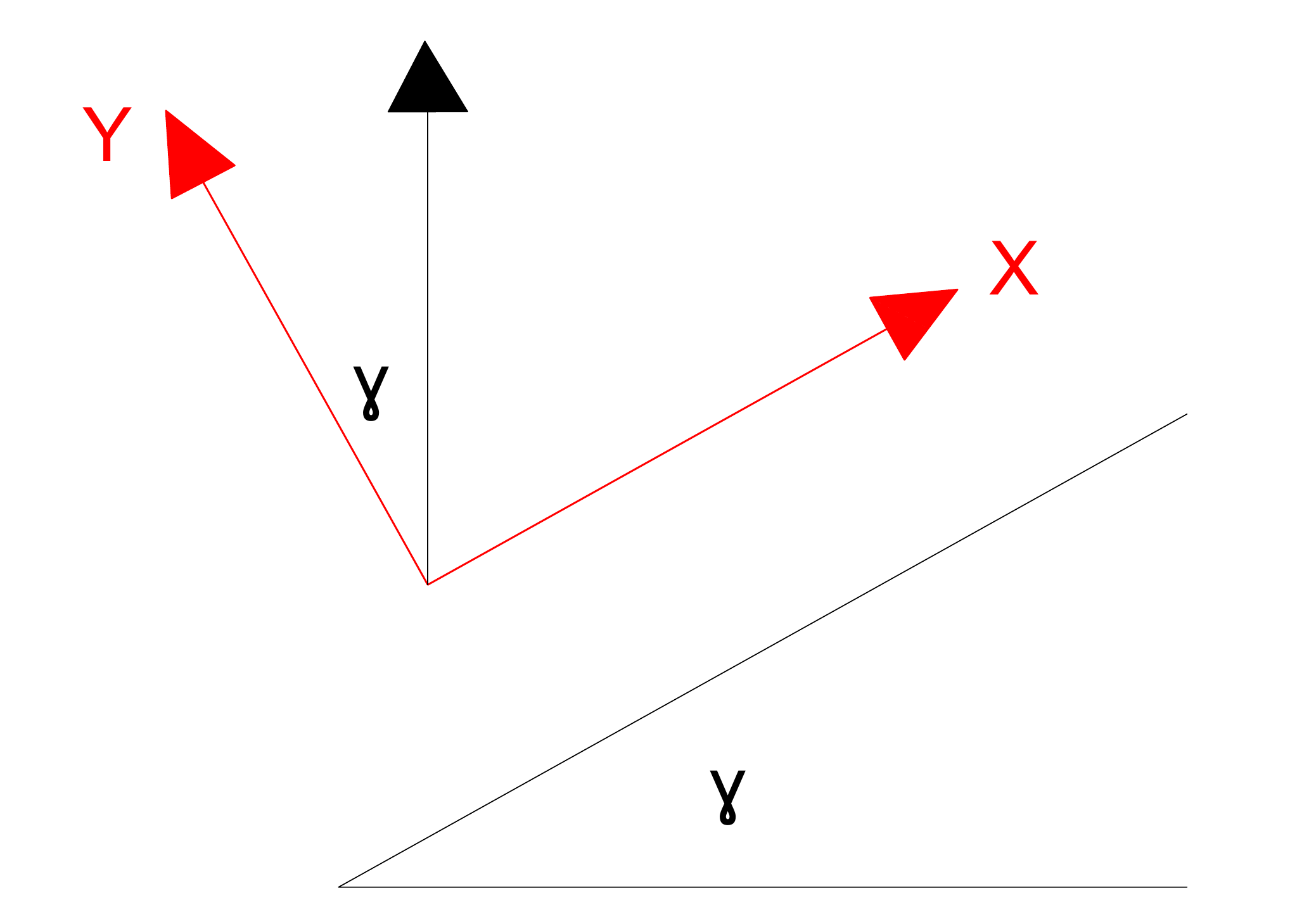}
\end{figure}
     We will denote by $(k,m)$ the corresponding Fourier variables, and  by $(u,w)$ the tangential and normal components of the velocity.
    In these coordinates, the fully nonlinear system \eqref{boussinesq} becomes
    \be
    \label{eq:NS-ep}
    \ba
    \d_t u + \delta \left(u\d_x + w\d_y\right)u - \sin \gamma b + \d_x p=0,\\
    \d_t w + \delta \left(u\d_x + w\d_y\right)w - \cos \gamma b + \d_y p=0,\\
    \d_t b + \delta \left(u\d_x + w\d_y\right)b + N^2 (\sin \gamma u + \cos \gamma w)=0,\\
    \d_x u + \d_y w=0,
\ea   
 \ee
    while    the linear inviscid wave system (\ref{wave})
   can be restated
   \begin{equation}
   \label{tilt-wave}
   \begin{aligned}
-  i\omega  \hat u - \sin \gamma \hat b + ik \hat P = 0,\\
- i\omega \hat  w - \cos \gamma \hat b + im \hat P =0,\\
- i\omega \hat  b + N^2\sin \gamma \hat u +N^2 \cos \gamma \hat w = 0,\\
ik \hat  u  + im \hat  w =0.
\end{aligned}
\end{equation}
Since it has to lift the boundary condition $ w_{|y = 0} = 0$, the reflected wave should have the same horizontal wave number $k$ and time frequency $\omega$ as the incident wave, but a different vertical wave number $m'$, which also satisfies the dispersion relation
\be\label{disp-relation}
\frac{1}{N^2}\omega^2= \frac{(k \cos \gamma - m \sin \gamma)^2}{k^2 + m^2}=\sin^2 \beta.
\ee
 The equation for $m$ and $m'$  is therefore
\be\label{eq:m'}
m^2(\sin^2 \beta - \sin^2 \gamma)+ 2 k m  \cos \gamma \sin \gamma + (\sin^2 \beta-\cos^2 \gamma) k^2 = 0.
\ee
The roots are  $$m= k {\cos \gamma \sin \gamma - \cos \beta \sin \beta\over \sin^2 \beta - \sin^2 \gamma}, \qquad m' = k {\cos \gamma \sin \gamma + \cos \beta \sin \beta\over \sin^2 \beta - \sin^2 \gamma}.$$ 
Note in particular that, as $\beta\to \gamma$, $|m'| \to \infty$. 

The elementary solution of (\ref{tilt-wave}) supplemented with the zero flux condition is therefore of the form
$$
\ba
\begin{pmatrix}
\hat u \\ \hat w \\ \hat b 
\end{pmatrix}
& =  A\left(\begin{array}{c}
1 \\
-\frac{k}{m}\\
N\frac{i(k\cos \gamma-m \sin \gamma)}{m \sin \beta} 
\end{array}\right)e^{-i N\sin \beta t + i k x + i my}\\ 
&+ 
B\left(\begin{array}{c}
1 \\
-\frac{k}{m'}\\
N\frac{i(k\cos \gamma-m' \sin \gamma)}{m' \sin \beta} 
\end{array}\right)e^{-i N \sin \beta t + i k x + i m' y} +c.c.,
\ea$$
where $c.c.$ denotes the complex conjugate,
with the polarization condition 
$$A\frac{k}{m}+B \frac{k}{m'}=0,$$
meaning that the amplitude of the reflected wave is $O(\frac{m'}{m})$.

The most effective situation for boundary mixing arises therefore when an oncoming
wave reflects off a bottom slope with angle $\gamma$ which nearly matches the angle of wave propagation, namely $\beta = \gamma -\eps^2$ with $\eps \ll 1$.
At the critical angle $\beta =\gamma$, the analytic theory of internal waves reflecting off a
uniformly sloping bottom  predicts that the reflected wave has infinite
amplitude and infinitesimal wavelength. These unphysical results signal the failure
of the idealizations (namely, linear waves and inviscid fluid).

\bigskip
$\bullet$ {\bf The viscous and nonlinear effects} associated to the near-critical reflection of internal waves have been studied by Dauxois and Young \cite{DY} in 1999.
The main idea is that, even though the original perturbation generating the incident wave is small, the amplitude of the reflected wave is enhanced by the focusing mechanism,
so that the  nonlinear coupling between the incident and the reflected wave might no longer be negligible.
Moreover, the dissipation is enhanced because of the fact that the normal wavenumber of the reflected wave is much bigger than the original wavenumber.
  
  Actually, when one considers the viscous case, the analysis is much more delicate as the system of equations has to be supplemented with more boundary conditions~: since the equations for $u$, $w$ and $b$ are parabolic, one should impose one boundary condition for each one of these quantities. From the physical point of view, it is natural to prescribe a Dirichlet boundary condition on both components of the velocity (no slip condition), and a Neumann condition on the buoyancy (no diffusive flux through the slope, see \cite{DY} for a discussion on this point). Even at the linear level, i.e. for the viscous wave system
  \begin{equation}
\label{viscous-wave}
\d _t \begin{pmatrix} \hat u_{p}\\ \hat w_{p} \\ \hat b_p\end{pmatrix} + \begin{pmatrix}
 \nu (k^2+m^2) & 0 & {m(k \cos \gamma - m \sin \gamma) \over k^2 +m^2}  \\  0 & \nu (k^2+m^2) & {k(m\sin \gamma - k \cos \gamma) \over k^2 +m^2} \\
N^2\sin \gamma  & N^2\cos \gamma &   \kappa (k^2+m^2)  \end{pmatrix} \begin{pmatrix} \hat u_{p}\\ \hat w_{p} \\ \hat b_p\end{pmatrix}  = 0\,,
\end{equation}
 with $p=(k,m)$, it is therefore much more complicated to compute elementary solutions satisfying the suitable boundary conditions since there are   three matching conditions instead of one. 
  In their paper, Dauxois and Young actually discarded the boundary condition on the buoyancy (which is justified for instance if $\kappa = 0$), and chose the  distinguished scaling $\nu= \eps^6$. They found that elementary solutions can be then decomposed as the sum of 
  \begin{itemize}
  \item the incident wave;
  \item two boundary layer terms with exponential decay in $y/\eps^2$.
  \end{itemize}
  In other words, this means that one can find two values $m_1, m_2$ (distinct from $m$, and having a non negative imaginary part) such that the matrix in (\ref{viscous-wave}) has $i\omega = iN\sin \beta$ as an eigenvalue. These two eigenvalues $m_1$, $m_2$ bifurcate from $m'$ when we add the small viscosity $\nu = O(\eps^6)$, and they are  of order $O(1/\eps^2)$ like $m'$. The matching conditions for $u$ and $w$ at $y= 0$ provide then polarization conditions for the amplitudes of these two damped reflected waves.
  
  In the boundary layer, for $\delta$ small enough, we expect the nonlinear term to be weak so that it should not destroy the wave structure. This means that the nonlinearity can be dealt with as a weak perturbation of the linear evolution. At leading order, we then expect it will produce two new modes, one which is oscillating with frequency $2\omega$ and a mean flow. The amplitudes of these modes should be a small correction to the boundary layer, but which has non zero trace on the boundary! In their paper, Dauxois and Young lifted this (small) remaining trace by adding
  \begin{itemize}
  \item a second harmonic propagating in the outer domain provided that $\frac{2\omega}{N} \le 1$;
  \item a mean flow (which is not written explicitly as it is expected to stay localized close to the boundary).
  \end{itemize}
As a result of this construction, they obtain an approximate solution in the sense that it satisfies the equation (\ref{boussinesq}) with suitable scalings for the parameters of the system, together with the boundary conditions $u_{|y = 0}  = w _{|y = 0} = 0$, up to remainders which should be either smaller in energy, or  of same order but fast oscillating (thus converging weakly to 0). We will say that the approximation is consistent. This Ansatz shows in particular that there is a critical amplitude where the nonlinearity, although small, is no longer negligible even in the outer domain. This is made evident by the generation of the second harmonic.

\bigskip
$\bullet$ {\bf This scenario has been validated by
experimental visualization} of the reflection process.
The lab experiments conducted by physicists, especially in the group of   Dauxois \cite{DDF, GDDSV}, have indeed confirmed that the nonlinearity plays a key role in the boundary layer.

Dauxois, Didier, and Falcon  introduced  in \cite{DDF} the Schlieren technique  to study the spatiotemporal
evolution of the internal waves reflection close
to the critical reflection. The internal wave, producing density disturbances,
causes lines to distort, this distorting line pattern
being recorded by a camera. Note that this experimental
technique is sensitive to the index gradient, and therefore to
the density gradient. The dynamics of isopycnals is then found in good qualitative agreement with
the theory. Moreover, this experiment confirms
the theoretically predicted scenario for the transition to
boundary-layer turbulence responsible for boundary mixing:
the growth of a density perturbation produces a statically
unstable density field which then overturns with small-scale
fluctuations inside.

Peacock and Tabaei then presented  in \cite{PT} the first set of experimental visualizations (using also the digital Schlieren method) that confirm the existence of radiated higher-harmonic beams.
  For arrangements in which the angle of propagation of the second harmonic exceeds the slope angle, radiated beams are visualized. When the propagation angle of the second harmonic deceeds the slope angle no radiated beams are detected, as the associated density gradient perturbations are too weak for the experimental method. The case of a critical slope is also reported.

Quantitative results have finally been obtained in \cite{GDDSV}.  Experiments were carried out in the  Coriolis
platform, in Grenoble, filled with salted water. The large
scale of the facility allows to strongly reduce the viscous
dissipation along wave propagation and quantitative
results are obtained thanks to high-resolution Particle Image Velocimetry measurements.
Generation of the second and third harmonic frequencies is clearly demonstrated in the impact zone.
Although these harmonics are almost invisible from the
instantaneous velocity field, they are very clearly apparent
after the filtering procedure. These experiments also provide evidence that
harmonics with frequency higher than $N$ cannot propagate
and remain trapped near the slope.

\subsection{Mathematical description of the near-critical reflection}

The mathematical analysis of the near-critical reflection follows essentially the same lines, but requires a more careful treatment of some delicate points which we will explain now.

First of all, {\bf the consistency of the approximation is not sufficient} to deduce that the solution of the original problem will be close to the approximate solution. This issue  is related to the possible instabilities of the system~: a small error on the equation (in the form of a source term $R_{app}$) could then generate important deviations of the dynamics. This means that, in order to prove that the matched asymptotic expansion $\cW_{app}$  provides a good approximation of what happens in reality, one needs to prove some stability for the original weakly nonlinear system (\ref{boussinesq}). 
The classical way of doing so is to establish energy inequalities~: for the Navier-Stokes equations, we typically expect that the growth of the $L^2$ norm $\| \cW- \cW_{app}\|_{L^2}$ should be controlled by the Lipschitz norm $\delta \| \nabla \cW_{app}\|_{L^\infty}$, provided that there is no error (at all) on the boundary conditions. 
\be\label{inequality:energy}
\begin{aligned}
 \| (\cW-\cW_{app}) (t) \|_{L^2}^2  \leq &  \| (\cW-\cW_{app}) (0) \|_{L^2}^2 \exp (\delta \| \nabla \cW_{app}\|^2_{L^\infty} t ) \\
 & +\int_0^t  \| R_{app}(s) \|_{L^2}^2  \exp (\delta \| \nabla \cW_{app}\|^2_{L^\infty} (t-s)  ) ds\,.
 \end{aligned}
 \ee
 However this has many implications. More details will be provided in Section \ref{subsec:stability-inequality}.
\begin{itemize}
\item[(i)] We need to work with solutions of finite energy (or at least such that $\cW-\cW_{app}$ has finite energy), which is not the case of (a finite sum of) plane waves!
\item[(ii)] The remainder has to be small (in energy), and actually smaller than the second harmonic if we would like to prove that this term is physically relevant.
\item[(iii)] The time interval on which the approximation is accurate depends on  the Lipschitz (or at least the $L^\infty$) norm of the approximate solution.
\end{itemize}
Points (ii) and (iii) are somehow technical, they can be tackled by tracking the dependency of $\cW_{app}$ and $R_{app}$ with respect to the different parameters. We will also need to construct some additional correctors to lift the boundary conditions. We refer to Section \ref{sec:NL-BL} for these technical details.

Point (i) is more tricky. It is related to the fact that  a ``plane wave" is not a good physical object (although it is considered as a very basic object in physics!) This has already been mentioned in the introduction when talking about propagation of wave and group velocity. These notions do not make sense for a single plane wave, but rather for a wave packet. Plane waves just provide  a tool to decompose these wave packets on elementary objects (just like the Fourier transform).  A very pedagogical  discussion on these topics can be found for instance in \cite{NLM}. In all the sequel of this paper, we will thus consider wave packets, with energy density concentrated close to the frequency $\omega_0 = \sin \gamma$ in order to see the effects of criticality. Of course, as long as we study linear equations, looking at a wave packet (i.e. at a superposition of plane waves) does not introduce any additional difficulty for the matched asymptotic expansion. But this is no longer true when the nonlinearity enters into the game.

\bigskip
Let us then go back to the {\bf construction of the approximate solution} $\cW_{app}$, as there are also here some points which need to be clarified.
\begin{itemize}
\item[(iv)] In the construction by Dauxois and Young \cite{DY}, one boundary condition (the one on the buoyancy $b$) is discarded, supposedly because it should produce effects of higher order. This is however not completely clear from the computations given in the paper (see the comments p. 283). Section \ref{sec:linear_BL} of the present paper will provide a very systematic construction for linear boundary layers. We refer to \cite{GVP, DSR} for a presentation of the method. We will see in particular that the solution to the viscous wave equation (\ref{viscous-wave}) involves a superposition of boundary layers of different sizes, and that there is no overdetermination of the problem if we take into account all these contributions.
\item[(v)] Since we will consider wave packets, we will need to understand how to deal with the nonlinear interactions in the boundary layer for a general superposition of waves (see Section \ref{sec:NL-BL}). Section \ref{sec:NL-BL} will provide a careful analysis of the generation of both the rectified (or mean) flow and the second harmonic.

\item[(vi)] The rectified flow  in \cite{DY} (defined as the component with $\omega=0$ generated by the weak nonlinearity) is neglected since it has to vanish far from the boundary (see p. 282). Although in \cite{DY} there is no quantitative estimate of the decay rate, nor of the energy contained in this rectified flow (which could be more energetic than the second harmonic for instance), this is in agreement with our construction. We have indeed two boundary layer mean flows (vanishing far from the boundary), and the third boundary condition can be lifted in a very crude way (by adding a small corrector which has no physical relevance).

\end{itemize}

\subsection{Notation}

Our construction of an approximate solution will involve many terms, which all play a different role. For the reader's convenience, we have gathered here the notations that will be used throughout the paper, and by doing so we roughly sketch the main steps of the construction.

\begin{itemize}
	\item We denote by $\cW$ the vector $\begin{pmatrix}
u\\w\\b
	\end{pmatrix}$;
	
	\item The first step of our construction is the definition of a solution of the linear problem. This solution will be the main order term in the approximate solution, and therefore we will denote it by $\cW^0$;
	
	\item The first order correctors, which lift the nonlinearity $\delta Q(\cW^0, \cW^0)$, will be denoted by $\cW^1$.

\end{itemize}

The notations $\cW^j$ allow us to distinguish between terms of different orders within the approximate solution. But we also need to specify the nature of each term within a given order. For instance, $\cW^0$ will be the sum of an incident wave packet, of a boundary layer term localized within a region of width $\eps^2$ of the boundary $y=0$, and of a boundary layer term localized within a region of width $\eps^3$ of the boundary $y=0$. Therefore we adopt the following notation:

\begin{itemize}
\item 	The incident wave packet will be denoted $\Winc$;

\item The boundary layer terms localized within a region of width $\eps^2$ (resp. $\eps^3$) of the boundary will be denoted $\WBLII^j$ (resp. $\WBLIII^j$), where the superscript $j$ refers to the order of the term they belong to;

\item The mean flow will be denoted by $\WMF$;

\item The second harmonic will be denoted by $\WSH$.
\end{itemize}

\subsection{The stability result}\label{subsec:stability}
From now on, we choose $N=1$ in order to alleviate the notation.\\
We consider the Cauchy problem associated to the Boussinesq equations in $\R^2_+=\R\times \R^+$,
\begin{equation}
\label{original_system}
\ba
\partial_t u - b \sin \gamma + \partial_x p + \delta(u \partial_x u + w \partial_y u)=\nu_0 \varepsilon^6 \Delta u, \\
\partial_t w - b \cos \gamma + \partial_y p + \delta(u \partial_x w + w \partial_y w)=\nu_0 \varepsilon^6 \Delta w, \\
\partial_t b + u \sin \gamma + w \cos \gamma + \delta(u \partial_x b + w \partial_y b)=\kappa_0 \varepsilon^6 \Delta b, \\
\partial_x u + \partial_y w=0,
\ea
\end{equation}
endowed with the boundary conditions
\begin{equation}
\label{BC}
u_{|y = 0} = w_{|y = 0} = \d_y b_{|y = 0} = 0\,.
\end{equation}
Note that this is a just rewriting of the Boussinesq system (\ref{boussinesq}) in the slope coordinates, nothing but the viscous version of system (\ref{eq:NS-ep}). As discussed at the beginning of Section \ref{sec:NL-BL}, after a complete linear analysis of all the possible regimes induced by the different parameters of the system in Section \ref{sec:linear_BL}, the nonlinear system will be treated according to the scalings by Dauxois and Young in \cite{DY}, where, in particular, the size of the viscosity is $\nu=\nu_0 \varepsilon^6, \, \kappa=\kappa_0 \varepsilon^6$.\\
Our main result is a stability estimate for the approximate solution $\cW_{app}=(u_{app}, w_{app}, b_{app})$, i.e. an $L^2$ estimate of the difference between $\cW_{app}$ and the weak solutions $\cW=(u, w, b)$ to (\ref{original_system})-(\ref{BC}). Before providing the stability theorem, the global existence of weak solutions to the Cauchy problem for system (\ref{original_system}) with boundary conditions (\ref{BC}) is stated below. \\
We will use the following notation.
\be\label{def:functional-spaces}
\ba
&
 \V_\sigma:=\{ (u,w,b)\in H^1(\R^2_+)^3,\ u_{|y=0}=w_{|y=0}=0,\  \partial_x u + \partial_y w=0\},\\
&\V_\sigma'=\text{dual space of } \V_\sigma. \\
\ea
\ee
\begin{Prop}\label{prop:weak-solutions}
Let $\cW_0=(u_0, w_0, b_0)$ a divergence free $L^2$ initial data. Then there exists a unique global weak solution $$\cW \in C(\R^+; \V_\sigma') \cap L^\infty(\R^+; L^2(\mathbb{R}^2_+)) \cap L^2_{loc}(\R^+; \V_\sigma)$$ to system (\ref{original_system})-(\ref{BC}), which satisfies the following energy inequality for all $t \ge 0$,
\be \label{energy-inequality}
\ba
\|\cW(t)\|^2_{L^2(\R^2_+)} + 2 \varepsilon^6 \nu_0 \int_0^t \Bigg(\|\nabla u(s)\|^2_{L^2(\R^2_+)}+ \|\nabla w(s)\|_{L^2(\R^2_+)}^2\Bigg) \, ds\\
+2\varepsilon^6 \kappa_0 \int_0^T \|\nabla b(s)\|^2_{L^2(\R^2_+)} \, ds
\le \|\cW_0\|^2_{L^2(\R^2_+)}.
\ea
\ee
\end{Prop}
The proof of this theorem follows by adapting the result of global weak (Leray) solutions to the incompressible Navier-Stokes equation in general domains in \cite{CDGG}. The crucial point is the conservation of energy, represented by the energy inequality (\ref{energy-inequality}), which can be simply obtained in the usual way, by taking the scalar product of (\ref{original_system}) with $\cW$.\\
For the sake of completeness, we will provide a sketch of the argument in the Appendix.\\
We now  state our main result.
\begin{Thm}\label{Thm:stability}
Consider the Boussinesq equations (\ref{original_system}) in the scaling by Dauxois and Young in $\mathbb{R}_+^2=\mathbb{R} \times \mathbb{R}_+$, with boundary conditions (\ref{BC}).\\ Then there exists a vector field 
$$\cW_{app}:=(u_{app}, w_{app}, b_{app})=\Winc+\cW_{BL}+\cW^1_{II}+\cW_{corr},$$
where $\Winc, \,\cW_{BL}, \, \cW^1_{II}, \, \cW_{corr}$ are respectively an incident wave packet, a boundary layer, a second harmonic wave packet and an additional correction term, which is an approximate solution, in the sense that
\begin{align*}
\ba
\partial_t u_{app} - b_{app} \sin \gamma + \partial_x p_{app} + \delta(u_{app}, w_{app}) \cdot \nabla u_{app}-\nu_0 \varepsilon^6 \Delta u_{app}=O(\delta \varepsilon^2), \\
\partial_t w_{app} - b_{app} \cos \gamma + \partial_y p_{app} + \delta(u_{app}, w_{app}) \cdot \nabla w_{app}-\nu_0 \varepsilon^6 \Delta w_{app}=O(\delta \varepsilon^2), \\
\partial_t b_{app} + u_{app} \sin \gamma + w_{app} \cos \gamma + \delta(u_{app}, w_{app}) \cdot \nabla b_{app}-\kappa_0 \varepsilon^6 \Delta b_{app}=O(\delta \varepsilon^2), \\
\partial_x u_{app} + \partial_y w_{app}=0,
\ea
\end{align*}
where the remainders $O(\delta \varepsilon^2)$ have to be understood in the sense of the $L^2(\mathbb{R}^2_+)$ norm, and 
endowed with the boundary conditions
\begin{align*}
{u_{app}}_{|y = 0} = {w_{app}}_{|y = 0} = \d_y {b_{app}}_{|y = 0} = 0\,.
\end{align*}

Furthermore, denoting by $\cW$ the unique weak solution to the Cauchy problem associated with system (\ref{original_system})-(\ref{BC}), with initial data 
$$\cW_0=\cW_{app}(t=0),$$
we have the following stability estimate:
\begin{equation}\label{inequality:stability}
\|(\cW_{app}-\cW)(t)\|_{L^2(\R^2_+)} \le \delta \varepsilon^2 \exp((\delta \varepsilon^{-2}+1) t).
\end{equation}
\end{Thm}
\begin{Rmk}
Let us add some comments on Theorem \ref{Thm:stability}.
\begin{itemize}
\item The first part of the statement could be reformulated saying that $\cW_{app}$ is a \textbf{consistent approximate solution}, which means that it satisfies the equations of system (\ref{original_system})-(\ref{BC}) modulo a remainder (of order $\delta \varepsilon^2$ in $L^2$). In the second part, we infer that the consistent approximate solution is also \textbf{stable in the sense of the $L^2$ norm}, meaning that the difference in $L^2$ between this approximate solution and the unique weak solution to the original system (\ref{original_system})-(\ref{BC}) with initial data $\cW_{app}(t=0)$ is smaller than a remainder of size $\delta \varepsilon^2$ in $L^2$. Note that, as widely discussed in the rest of the paper, the boundary conditions (\ref{BC}) need to be exactly satisfied by the approximate solution in order to establish the stability inequality (\ref{inequality:stability}). A consistent approximate solution is indeed already given by the sum of first three terms $\Winc+\cW_{BL}+\cW_{II}^1$, while to achieve stability we need to add another corrector $\cW_{corr}$, whose explicit expression is given in Proposition \ref{prop:W1}.
\item We point out that far from the boundary, the leading order term of the approximate solution $\cW_{app}$ is represented by the incident wave packet $\Winc$. 
\item Theorem \ref{Thm:stability} could also be stated in a more general fashion, by taking into account any incident wave packet (not only the ones of the critical regime by Dauxois and Young). The construction of the approximate solution can be done exactly in the same spirit and it should be easier, the critical case being the most difficult one to handle.
\end{itemize}
\end{Rmk}

More details on the terms involved in the expression of the approximate solution $\cW_{app}$ are provided in Lemma \ref{lem:size-W0} and Proposition \ref{prop:W1}.\\
The proof of Theorem \ref{Thm:stability} relies on a quite accurate quantification of all the sizes of the terms and the remainders of the approximate solution in $L^2$ and $L^\infty$. These computations will be provided in details in Section \ref{sec:NL-BL}. A general construction of the profiles and sizes of the boundary layers, depending on their time frequency and tangential wave number, is provided in Section \ref{sec:linear_BL}. The stability inequality is established in Section \ref{sec:stability}.


\section{Linear viscous boundary layers}
\label{sec:linear_BL}

The purpose of this section is to provide a systematic description of boundary layers in the linear case. In particular, we will explain how boundary layer sizes and profiles can be computed, which boundary conditions can be lifted, and we will derive the asymptotic behavior of the boundary layer sizes in different regimes, including the case of critical reflection. In this case, we will also derive the expansions of the linear boundary layer operators, that will allow us to perform the weakly nonlinear analysis of Section \ref{sec:NL-BL}.

As explained in the introduction, the first item in the construction of an approximate solution is an incident wave $(u_i, w_i, b_i)$ (or an incident wave packet, which is an infinite linear superposition of incident waves). This incident wave does not solve the boundary conditions
\be\label{original_BCs}
u=0, \; w=0, \; \partial_yb=0 \quad \text{on} \quad y=0,
\ee
and it is expected that boundary layers take place close to the wall $y=0$ to lift the traces of $(u_i, w_i, \d_y b_i)$. Therefore the purpose of our analysis in the present section is the following: we seek an exact solution of the linear system
\be\ba
\d_t u - \sin \gamma b + \d_x p = \nu \Delta u,\\
\d_t w - \cos \gamma b + \d_y p = \nu \Delta w,\\
\d_t b + \sin \gamma u + \cos \gamma w = \kappa \Delta b,\\
u_x + w_y=0,
\ea 
\label{lin-system}
\ee
endowed with the boundary conditions
\be
\label{cond-BL}
\ba
u_{|y=0}=\mathfrak u\exp(i(kx-\omega t)), \\
 w_{|y=0}=\mathfrak w\exp(i(kx-\omega t)),\\
\d_y b_{|y=0}=\mathfrak b  \exp(i(kx-\omega t)),
\ea
\ee
where $k,\omega\in \R$ are prescribed, and $\mathfrak u, \mathfrak w, \mathfrak b \in \mathbb C$. We refer to the work of G\'erard-Varet and Paul  \cite{GVP} for a presentation of the methodology developed here, and to the paper \cite{DSR} by the last two authors for an application to a case in which several boundary layers with different sizes co-exist. Since equation \eqref{lin-system} is linear and has constant coefficients, it is natural to look for modal solutions, i.e. (linear combinations of) functions of the form 
\be\label{modal}
(u,w,b,p)(t,x,y)=(U,W,B,P)\exp(i(kx-\omega t)-\lambda y),
\ee
with $\Re(\lambda)>0$. When we plug this expression into system \eqref{lin-system}, we obtain
\[
A_{\nu,\kappa}(\omega,k,\lambda)\begin{pmatrix}
U\\W\\B\\P
\end{pmatrix}=0,
\]
where
\be\label{def:matrix-A}
 A_{\nu,\kappa, \omega, k}(\lambda)=\begin{pmatrix}
-i\omega + \nu (k^2 - \lambda^2)  & 0 &-\sin \gamma & ik\\
0 & -i\omega + \nu (k^2 - \lambda^2)  &-\cos \gamma & -\lambda\\
\sin \gamma & \cos \gamma & -i\omega + \kappa (k^2 - \lambda^2)  & 0\\
i k & - \lambda &0&0
\end{pmatrix}.
\ee
As a consequence,  a modal function of the form \eqref{modal} is a non trivial solution of \eqref{lin-system} if an only if $\det  A_{\nu,\kappa,\omega,k}(\lambda)=0$ and $(U,W,B,P)\in \ker  A_{\nu,\kappa,\omega,k}(\lambda)\setminus \{0\}$. The corresponding solution is a boundary layer mode if $\Re(\lambda)\gg 1$, and the size of the boundary layer is then $(\Re(\lambda))^{-1}$. Furthermore, the number of boundary conditions that can be lifted by the boundary layer is equal to the dimension of the vector space
\begin{multline*}
\mathrm{Vect}\ \{ (U,W,-\lambda B)\in \mathbb C^3,\; \exists (P, \lambda) \in \C^2,\; \Re(\lambda)>0\text{ and }\\ \det  A_{\nu,\kappa,\omega,k}(\lambda)=0,\  (U,W,B,P)\in \ker  A_{\nu,\kappa,\omega,k}(\lambda) \},
\end{multline*}
where $(U, W)$ lift the Dirichlet boundary conditions for the velocity field, while $-\lambda B$, which is the amplitude of the derivative in $y$ of expression (\ref{modal}), counts for the no flux condition on the buoyancy $\partial_y b=0$ at $y=0$.
Now, it can be easily checked that $\det A_{\nu,\kappa}(\omega,k,\lambda)$ is a polynomial of degree 6 in $\lambda$, that has therefore 6 complex roots, counted with multiplicity. We now investigate the asymptotic values of these roots depending on the viscosity and the slope criticality. We will prove in particular that we can usually lift three boundary conditions thanks to this linear boundary layer.

\subsection{Sizes of the boundary layers in terms of the viscosity and slope criticality}

A straightforward computation leads to
\begin{eqnarray*}
&&\det A_{\nu,\kappa,\omega,k}(\lambda)\\&=&(k^2 - \lambda^2)\left(- i \omega + \nu (k^2 - \lambda^2)\right)\left(- i \omega + \kappa (k^2 - \lambda^2)\right)\\
&&-2 i k \lambda \sin \gamma \cos \gamma + k^2 \cos^2 \gamma - \lambda^2 \sin^2 \gamma\\
&=&\nu \kappa (k^2-\lambda^2)^3 - i \om (\nu + \kappa) (k^2 - \lambda^2)^2+(\omega^2 - \sin^2 \gamma )\lambda^2\\
&&-2ik\lambda \sin\gamma \cos \gamma+ k^2 (\cos^2 \gamma - \omega^2)\\
&=& - \nu \kappa \lambda^6 + (-i\omega (\kappa + \nu) + 3 \nu \kappa k^2) \lambda^4  \\
& &+ \left[\omega^2 - \sin^2 \gamma  +2i\om(\kappa + \nu) k^2- 3 \nu \kappa k^4 \right]\lambda^2\\
&& -2 i k \sin \gamma \cos \gamma \lambda+k^2(\cos^2 \gamma - \omega^2 - i \omega (\kappa+\nu )k^2 + \nu \kappa k^4). 
\end{eqnarray*}
Note that in the case $\nu=\kappa=0$ and $\lambda=-im$, we retrieve the dispersion relation \eqref{disp-relation}.
In all regimes considered below, we will make the following assumptions: $0<\nu \ll 1$, $0<\kappa \ll 1$, $ |k|\leq c_0$, $c_0^{-1}\leq |\cos^2 \gamma - \omega^2|\leq c_0$ for some fixed constant $c_0$. We introduce a criticality parameter $\zeta:=\omega^2 - \sin^2 \gamma$, and we will investigate both the cases $|\zeta|\ll 1$ and $|\zeta|\gtrsim  1$. To simplify the discussion, 
we will also assume that $\nu $ and $\kappa$ are of the same order of magnitude. However, we will not systematically assume that we are in a critical setting. In other words, we will not assume any relationship between $\nu$ and $\zeta$.
It follows from the above assumptions that when $\omega$ and $k$ are bounded away from zero, the roots of  $\det A_{\nu,\kappa,\omega,k}$ are asymptotically equivalent to the ones of
\begin{align*}
 P_{\nu,\kappa,\omega,k}(\lambda):= - \nu \kappa \lambda^6 - i\omega (\kappa + \nu) \lambda^4 + (\zeta+2i\om(\kappa + \nu) k^2) \lambda^2\\- 2 i k \sin \gamma \cos \gamma \lambda + k^2(\cos^2 \gamma - \omega^2).
\end{align*}
However, for further purposes, we will also need to consider the case $|\omega|\ll 1$, $|k|\ll 1$, which we will do in a separate paragraph.

Let us now explain our general strategy to find the boundary layer sizes. We first look for roots with the smallest possible absolute value. In the case $|k|\geq c_0^{-1}$, since all coefficients of $P_{\nu,\kappa,\omega,k}$ are bounded and the coefficient of order zero is bounded away from zero, the smallest roots will be of size 1, and all coefficients of order $o(1)$ can be neglected. We then look for roots $\lambda$ with a large absolute value, balancing the higher order terms, whose coefficients are smaller, with (some of) the lower order ones. This leads to an Ansatz of the form $|\lambda| \propto \nu^{-\alpha}$ for some $\alpha>0$. We then plug this Ansatz into $P_{\nu,\kappa,\omega,k}$, check that the terms that have not been taken into account can indeed be neglected, and derive an effective equation for $\lambda$. While looking for possible asymptotic behaviors for $\lambda$, it is also useful to keep in mind that the product of all roots is $-(\nu\kappa)^{-1}(\cos^2 \gamma - \omega^2)$.

\subsubsection{Case $|k|\geq c_0>0$, $|\omega|\geq c_0>0$}

\paragraph{Case $|\zeta|\gtrsim 1$ (non-critical reflection)}In this case, it can be easily checked that the roots of  $P_{\nu,\kappa,\omega,k}$ can be classified in the following way:
\begin{itemize}
\item There are two roots of size 1, which therefore do not correspond to boundary layers, and which are approximate roots of the equation
\[
(\omega^2 - \sin^2 \gamma) \lambda^2 - 2 i k \sin \gamma \cos \gamma \lambda + k^2 (\cos^2 \gamma - \omega^2)=0.
\]
These two roots converge towards pure imaginary values in the limit $\nu, \kappa \to 0$ when $|\omega|$ remains below the  critical value $\omega_c=1$. They then correspond to the incident and the reflected wave in the non-critical case (compare the above equation with the dispersion relation \eqref{disp-relation} or with \eqref{eq:m'}). When $|\omega|>1$, the two roots have a non-zero real part: one of them (say $\lambda_1$) has a negative real part (and must therefore be discarded), and the other (say $\lambda_2$) has a positive real part. It corresponds to an exponentially decaying mode, which is however not confined to a boundary layer.

\item There are four roots $\lambda$ such that $|\lambda |=O( \nu^{-1/2})$ (remember that we have assumed that $\nu$ and $\kappa$ are of the same order), which are approximate roots of the equation
\[
-\nu \kappa \lambda^4 - i \omega(\kappa + \nu) \lambda^2 + \zeta =0.
\]
Among these roots,  two have a strictly positive real part, and two have a negative real part.

\end{itemize}

 \paragraph{Case $\nu^{1/3}\ll |\zeta|\ll 1$ (critical reflection with small diffusivity)}
In this case, there are four types of roots:
\begin{itemize}
\item There is one root of size 1, say $\lambda_1$, such that
\[
\lambda_1 \sim - i\frac{k (\cos^2 \gamma - \omega^2)}{2 \sin \gamma \cos \gamma}.
\]
This root corresponds to the incident wave.

\item There is one root of size $\zeta^{-1}$, say $\lambda_2$, which satisfies
\[
\lambda_2\sim \frac{2ik\sin \gamma \cos \gamma}{\zeta}.
\]
Plugging this expression back into the equation $A_{\nu,\kappa,\om,k}(\lambda_2)=0$, it can be proved that $\Re(\lambda_2)$ is negative and of size $\nu/\zeta^4$.

\item There are two roots of size $|\zeta/\nu|^{1/2}$, say $\lambda_3$ and $\lambda_4$, which are approximate solutions of
\[-i\om(\kappa +\nu) \lambda^2 + \zeta=0.
\]
Among these two roots, exactly one (say $\lambda_3$) has a strictly positive real part of size $|\zeta/\nu|^{1/2}$, and the other one has a strictly negative real part.

\item There are two roots of size $\nu^{-1/2}$, which are approximate solutions of 
\[
-\nu\kappa \lambda^2 - i \om (\kappa +\nu)=0.
\]
Among these two roots, exactly one (say $\lambda_5$) has a strictly positive real part of size $\nu^{-1/2}$, and the other one has a strictly negative real part.
\end{itemize}

 \paragraph{Case $\nu^{1/3}\sim  |\zeta|$ (critical reflection in the scaling of Dauxois and Young)}

This case is very similar to the one above. The only difference lies in the fact that the roots $\lambda_2, \lambda_3$ and $\lambda_4$ are now all of size $\nu^{-1/3}$, and are approximate solutions of
\be\label{eq:eigenroots-DY}
-i\om (\kappa + \nu) \lambda^3+ \zeta \lambda - 2ik \sin \gamma \cos \gamma=0.
\ee
One can check that this equation has exactly two roots with positive real part (see Remark \ref{rem:positive-real-part} below). 

 \paragraph{Case $|\zeta| \ll \nu^{1/3}$ (critical reflection with large diffusivity)}

Once again, this case is similar to the two cases above. The three roots $\lambda_2, \lambda_3$ and $\lambda_4$ are  all of size $\nu^{-1/3}$, and are now approximate solutions of
\[
i\om (\kappa + \nu) \lambda^3 - 2ik \sin \gamma \cos \gamma=0.
\]
The conclusion remains the same.

\begin{Rmk}[Number of roots with positive real part]
	In the analysis above, we have counted by hand the number of roots of $\det A_{\nu,\kappa,\omega,k}$ with positive real part in some specific cases. Les us now prove that this number remains constant as the parameter $\zeta$ varies. Assume that the parameters $\nu,\kappa,\gamma$ and $k$ are fixed, with $\nu$ and $\kappa$ small, and consider the roots $\lambda_i$, $i=1,...6$ as functions of the criticality parameter $\zeta$. Then it can be easily checked, using the previous arguments, that the curves $\zeta\mapsto \lambda_i(\zeta)$ do not cross. Indeed, if there were a crossing, it would obviously occur for two eigenvalues  within the same regime, and a quick look at each of the regimes described above ensures that this cannot happen. As a consequence, each function $\zeta\mapsto  \lambda_i(\zeta)$ is continuous (and even $\mathcal C^1$). Furthermore, the curves do not cross the imaginary axis: indeed, if a root $\lambda$ is imaginary, then we deduce that
	\[
	\Im (\det A_{\nu,\kappa,\omega,k}(\lambda))=0=\omega(\kappa + \nu) (\lambda^2 - k^2)^2
	\]
	and therefore $\lambda = \pm |k|$ and $\lambda\in \R$: contradiction. As a consequence, each curve $\zeta\mapsto \lambda_i(\zeta)$ remains in the same complex half-plane. In particular, $\lambda_2, \lambda_3$ and $\lambda_5$ always have a positive real part, and $\lambda_4 $ and $\lambda_6$ always have a negative real part. The eigenvalue $\lambda_1$ corresponds to the incident wave and is therefore discarded.
	
	\label{rem:positive-real-part}
\end{Rmk}

Note that an eigenvector corresponding to a root $\lambda$ is 
\be\label{def:eigenvector}
\begin{pmatrix}
	U_\lambda\\ W_\lambda\\ B_\lambda\\P_\lambda
\end{pmatrix}= \begin{pmatrix}
	1 \\ \frac{ik}{\lambda}\\ \frac{\sin \gamma +\frac{ik}{\lambda}\cos \gamma}{i\om - \kappa (k^2-\lambda^2)}\\
{\frac{1}{ik}\left[ i\om + \nu(\lambda^2 - k^2) + \sin \gamma \frac{\sin \gamma + i k \lambda^{-1} \cos \gamma}{i\om + \kappa (\lambda^2 - k^2)}\right]}
\end{pmatrix}
\ee
We now go back to each of the cases above, and check that the vectors
\[
\left\{ \begin{pmatrix}
U_{\lambda_j}\\ W_{\lambda_j}\\ - \lambda_j B_{\lambda_j}
\end{pmatrix},\quad j\in \{2,3,5\}  \right\}
\]
are linearly independent. As a consequence, {\bf we can always lift three boundary conditions.}

\subsubsection{Case $|k|\ll 1$, $|\omega| \ll 1$}

In the following sections, we will also need to investigate the case when $|k|$ and $|\omega|$ are small. More specifically, we will be interested in the scaling $|k| \lesssim \nu^{1/3}$, $|\omega| \lesssim \nu^{1/3}$. Going back to the expression of $\det A_{\nu,\kappa,\omega,k}$, we obtain the following asymptotic behaviors for the roots:
\begin{itemize}
	\item There are two roots of size $\nu^{1/3}$, say $\lambda_1$ and $\lambda_2$, such that if $|k| \gtrsim \nu^{1/3}$
	\[
	\lambda_j\sim \frac{-ik}{\tan \gamma},\quad j=1,2.
	\]
	Note that one of these roots (say $\lambda_1$) has a negative real part, the other one (say $\lambda_2$) a positive real part. However, the root $\lambda_2$ does not correspond to a boundary layer because $\Re(\lambda_2)\ll 1$. Computations lead to 
	\[
	\Re(\lambda_2)\sim - \Re(\lambda_1) \sim - \frac{(\nu + \kappa) |k|^3}{\sin^2 \gamma}=O(\nu^2)
	\]
	in the scaling $|k| \gtrsim \nu^{1/3}$.
	\item There are four roots of size $O(\nu^{-1/2})$, say $\lambda_i$ for $i\in \{3,\cdots, 6\}$ which  satisfy
	\[
	\lambda_i^4\sim - \frac{\sin^2 \gamma}{\nu\kappa}.
	\]
	Among these, two have a strictly positive real part (say $\lambda_3$ and $\lambda_5$), and two have a strictly negative real part.

\end{itemize}

Once again, we can theoretically lift three boundary conditions. Note however that the vector $X_{\lambda_2} \exp(-\lambda_2 z)$ has a very slow decay, which leads to several complications: first, if $|k|\ll \nu^{1/3}$ and $|\omega| \gtrsim \nu^{1/3}$, the asymptotic given for $\lambda_2$ above may become invalid, and computing an exact rate of decay in all cases becomes quite technical. Furthermore, in the rest of the paper, we will consider wave packets, i.e. superposition of plane waves. This entails that the boundary layer part of the solution corresponding to the rate $\lambda_2$ will behave like
\[
\int_{\R} \frac{1}{\nu^{1/6}} \varphi \left(\frac{k}{\nu^{1/3}}\right) \exp\left( i (k x - \omega t) - i \frac{k}{\tan \gamma} y - c\nu |k|^3 y\right) \:dk ,
\]
where $\varphi \in \mathcal C^\infty_0(\R)$.
The Plancherel theorem entails that the $L^2_{x,y}$ norm of this quantity is bounded by
\[\ba
\frac{1}{\nu^{1/6}}\left( \int_{\R}\int_0^\infty  \varphi^2 \left(\frac{k}{\nu^{1/3}}\right) \exp\left( -2 c \nu |k|^3 y\right) dy \:dk \right)^{1/2}\\ \leq \frac{1}{\nu^{1/6}} \left(  \int_{\R}\varphi^2 \left(\frac{k}{\nu^{1/3}}\right)  \frac{1}{2 c \nu |k|^3 } \, dk\right)^\frac{1}{2}.
\ea\]
It is not clear that the right-hand side is finite, unless we add further assumptions on $\varphi$. {\it Therefore, in the rest of this paper, we will discard the part of the solution that should be handled by the eigenvalue $\lambda_2$ in this regime. } As a consequence, we will only be able to lift two boundary conditions in this case.

\vskip3mm

Let us sum up the results of this paragraph in the following table.
\begin{table}[H]
\caption{Sizes of the boundary layers in different regimes}
\noindent\begin{tabular}{|p{2.4cm}|p{2.7cm} |p{2.7cm}|p{2.7cm}|}\hline \\
$|\zeta| \gtrsim 1$ & $\nu^{\frac{1}{4}} \lesssim| \zeta| \ll 1$ & {$\nu^{\frac{1}{3}}\ll |\zeta |\ll \nu^{\frac{1}{4}}$ }& $|\zeta |\lesssim \nu^{\frac{1}{3}}$\\ \\ \hline
\begin{itemize}[leftmargin=*]
	\item One reflected wave
	\item One BL of size $\nu^{1/2}$  {(lifting 2 conditions)}
\end{itemize}&
\begin{itemize}[leftmargin=*]
	\item One highly oscillating reflected wave (with slow decay)
	\item One BL of size $(\nu/\zeta)^{1/2}$
	\item One BL of size $\nu^{1/2}$
\end{itemize}
&
\begin{itemize}[leftmargin=*]
	\item One BL of size $|\zeta|^4/\nu$
	\item One BL of size $(\nu/\zeta)^{1/2}$
	\item One BL of size $\nu^{1/2}$
\end{itemize}
&
\begin{itemize}[leftmargin=*]
\item One BL of size $\nu^{1/3}$
\item One BL of size $\nu^{1/2}$  {(lifting 2 conditions)}
\end{itemize}
\\
\hline
\end{tabular}
\label{table:sec:2}
\end{table}

\begin{enumerate}
\item Case $|k|\gtrsim 1$:
\begin{itemize}
\item If $|\zeta| \gtrsim 1$, there is one reflected wave (lifting the vertical boundary condition) and one boundary layer of size $\nu^{1/2}$, lifting the two remaining boundary conditions;

\item If $\nu^{1/3}\ll |\zeta| \ll 1$, there are three boundary layers, of sizes $\zeta^4/\nu$, $(\nu/\zeta)^{1/2}$, and $\nu^{1/2}$;

\item If $|\zeta|\lesssim\nu^{1/3}$, there are two boundary layers, of sizes $\nu^{1/3}$ and $\nu^{1/2}$.
\end{itemize}
\item Case $| k |\lesssim \nu^{1/3},\ |\omega| \lesssim \nu^{1/3}$:\\

There is one boundary layer of size $\nu^{1/2}$, lifting two (out of three) boundary conditions.

\end{enumerate}

\subsection{Boundary operators}

We now gather the results and observations from the preceding paragraph. 

$\bullet$ Assume that we are in the regime $|k|\gtrsim 1$. Let $(\mathfrak u, \mathfrak w, \mathfrak b)\in \C^3$, and consider the linear system with unknowns $(a_2, a_3, a_5)$
\be\label{lin-syst-aj}
\ba
a_2 + a_3 + a_5 = \mathfrak u,\\
\frac{ik}{\lambda_2} a_2 + \frac{ik}{\lambda_3} a_3+ \frac{ik}{\lambda_5} a_5=\mathfrak w,\\
\sum_{j\in \{2,3,5\}}-\lambda_j a_j\left[\frac{\sin \gamma +\frac{ik}{\lambda}\cos \gamma}{i\om - \kappa (k^2-\lambda^2)}\right]=\mathfrak b.
\ea
\ee
Define the function
\be\label{def:sol-BL}
\begin{pmatrix}
u\\w\\b
\end{pmatrix}= \sum_{j \in \{2, 3, 5\}} a_j  \exp(i(kx-\omega t)-\lambda_j y) \begin{pmatrix}
U_{\lambda_j}\\ W_{\lambda_j}\ \\ B_{\lambda_j}
\end{pmatrix}.
\ee

Then by construction, we have the following result:

\begin{Lem}[Critical reflection for an oscillating flow]
	Assume that there exists a constant $c_0>0$ such that $c_0^{-1}\leq |k| \leq c_0$,  $c_0^{-1}\leq |\cos^2 \gamma - \omega^2| \leq c_0$. Assume that $|\zeta|\ll \nu^{1/4}$ and that $\nu\ll 1$, $\kappa\ll 1$ with $\nu/\kappa \propto 1$.

	Let $\mathfrak u, \mathfrak w, \mathfrak b \in \C$. Then there exists an exact solution $(u,w,b)$ of the linear system \eqref{lin-system} with:
	\[
	\ba
	u_{|y=0}=\mathfrak u \exp(i(kx-\omega t)),\\ w_{|y=0}=\mathfrak w \exp(i(kx-\omega t)) ,\\ \d_y b_{|y=0}=\mathfrak b  \exp(i(kx-\omega t)).
	\ea 
	\]
	
	This solution is given by \[
	\cB^c_{\omega, k}[\mathfrak u, \mathfrak w, \mathfrak b]:=  \sum_{j \in \{2, 3, 5\}} a_j  \exp(i(kx-\omega t)-\lambda_j y) \begin{pmatrix}
	U_{\lambda_j}\\ W_{\lambda_j}\ \\ B_{\lambda_j}
	\end{pmatrix},
	\]
	 where the coefficients $a_j$ are determined by the system \eqref{lin-syst-aj}. It is exponentially small outside a small boundary layer localized in the vicinity of $y=0$.

\label{lem:boundary-op-critical}
\end{Lem}

We have a similar result in the case of a non-critical reflection; the main difference lies in the fact that the solution is the sum of a boundary layer and of a flow that is not confined to a small region.

\begin{Lem}[Non-critical reflection]
	Assume that there exists a constant $c_0>0$ such that $c_0^{-1}\leq |k| \leq c_0$,  $c_0^{-1}\leq |\cos^2 \gamma - \omega^2| \leq c_0$. Assume that $|\zeta|\gtrsim 1$ and that $\nu\ll 1$, $\kappa\ll 1$ with $\nu/\kappa \propto 1$.

	Let $\mathfrak u, \mathfrak w, \mathfrak b \in \C$. Then there exists an exact solution $(u,w,b)$ of the linear system \eqref{lin-system} with:
	\[
	\ba
	u_{|y=0}=\mathfrak u \exp(i(kx-\omega t)),\\ w_{|y=0}=\mathfrak w \exp(i(kx-\omega t)) ,\\ \d_y b_{|y=0}=\mathfrak b  \exp(i(kx-\omega t)).
	\ea 
	\]
	
	This solution is given by
	\[\ba
	\cB^{nc,RW}_{\omega,k}[\mathfrak u, \mathfrak w, \mathfrak b]:= a_2  \exp(i(kx-\omega t)-\lambda_2 y) \begin{pmatrix}
	U_{\lambda_j}\\ W_{\lambda_j}\ \\ B_{\lambda_j}
	\end{pmatrix},\\
	\cB^{nc,BL}_{\omega,k}[\mathfrak u, \mathfrak w, \mathfrak b]:= \sum_{j \in \{3, 5\}} a_j  \exp(i(kx-\omega t)-\lambda_j y) \begin{pmatrix}
	U_{\lambda_j}\\ W_{\lambda_j}\ \\ B_{\lambda_j}
	\end{pmatrix},
\ea	\]
	 where the coefficients $a_j$ are determined by the system \eqref{lin-syst-aj}. It is composed of a boundary layer of size $\nu^\frac{1}{3}$ and of a (possibly decaying) reflected wave.

	\label{lem:boundary-op-noncritical}
\end{Lem}

$\bullet$ We now address the case of a  non-oscillating flow, i.e. $|k| \lesssim \nu^{1/3}$, $|\omega|\lesssim \nu^{1/3}$. As stressed in the previous subsection, in this case we only lift two boundary conditions. Looking at system \eqref{lin-syst-aj} and recalling that $\lambda_2= O(k)=O(\nu^{1/3})$, $\lambda_3, \lambda_5 =O(\nu^{-1/2})$, we see that the best approximation of the actual linear solution is to choose $a_2, a_3$ and $a_5$ so that $a_2=0$ and $a_3$ and $a_5$ satisfy
\be\label{lin-system-degenerate}
\ba
 a_3 + a_5 = \mathfrak u,\\
\sum_{j\in \{3,5\}}-\lambda_j a_j\left[\frac{\sin \gamma +\frac{ik}{\lambda}\cos \gamma}{i\om - \kappa (k^2-\lambda^2)}\right]=\mathfrak b.
\ea
\ee
We obtain the following result:

\begin{Lem}[Non-critical reflection for a  non-oscillating flow]
	 Assume that $\nu\ll 1$, $\kappa\ll 1$ with $\nu/\kappa \propto 1$, and that  $|k| \lesssim \nu^{1/3}$, $|\omega| \lesssim \nu^{1/3}$.

	Let $ \mathfrak u, \mathfrak b\in \C$. Then there exists an exact solution $(u,w,b)$ of the linear system \eqref{lin-system} such that
	\[
\ba
u_{|y=0}=\mathfrak u \exp(i(kx-\omega t)),\\ \d_y b_{|y=0}=\mathfrak b  \exp(i(kx-\omega t)).
\ea 
\]
	
	This solution is given by 
	\[
	\cB_{\omega, k}^{no}[\mathfrak u, \mathfrak w, \mathfrak b]:=\sum_{j \in \{3, 5\}} a_j  \exp(i(kx-\omega t)-\lambda_j y) \begin{pmatrix}
	U_{\lambda_j}\\ W_{\lambda_j}\ \\ B_{\lambda_j},
	\end{pmatrix}
	\]
	where the coefficients $a_3,a_5$ are determined by system \eqref{lin-system-degenerate}. Furthermore, the remaining trace satisfies
	\[
	w_{|y=0}= \exp(i(kx-\omega t)) \sum_{j \in \{3, 5\}} \frac{ik}{\lambda_j} a_j .
	\]

	\label{lem:boundary-op-non-osc}
\end{Lem}

\subsection{Superposition of  boundary layers in the Dauxois-Young scaling}

In the rest of this section, we focus on the case of a critical reflection in the scaling of Dauxois and Young \cite{DY}, i.e. we assume that $\nu^{1/3}\sim  |\zeta|$. We recall that in this case, there is a boundary layer of size $\nu^{1/3}$, which is described in \cite{DY}, and another one much smaller, of size $\nu^{1/2}$, which is absent from \cite{DY}. We now describe the interplay between these two boundary layers. More precisely, we compute asymptotic values for the coefficients $a_j$ defined in \eqref{lin-syst-aj}. We also switch to the notation that will be used in the rest of this paper, and that was introduced in \cite{DY}: more precisely, we take $\nu=\nu_0 \eps^6$, $\kappa=\kappa_0 \eps^6$, where $\eps\ll 1$ is a small parameter, and $\nu_0, \kappa_0>0$ are constant and independent of $\eps$. We therefore refer to the boundary layer of size $\nu^{1/3}$ (resp. $\nu^{1/2}$) as the $\eps^2$ (resp. $\eps^3$) boundary layer.

\begin{Def}
	In the rest of this section, we denote by $\bar \Lambda_2, \bar \Lambda_3, \bar \Lambda_5$ the complex numbers with positive real parts defined by
	\[
	\bar 	\Lambda_j:=\lim_{\nu\to 0}\lambda_j \nu^{1/3}= \lim_{\eps\to 0}\lambda_j \eps^2 \text{     for  } j=2,3,\quad \bar \Lambda_5:= \lim_{\nu\to 0}\lambda_5 \nu^{1/2}=\lim_{\eps\to 0}\lambda_5 \eps^3.
	\]
	
\end{Def}

A straight-forward linear analysis shows that there exists $\bar A_2, \bar A_3, \bar A_5 \in \C\setminus \{0\}$ such that
\[
a_2 \sim \frac{\bar A_2}{\eps^2},\quad a_3 \sim \frac{ \bar A_3}{\eps^2},\quad a_5 \sim \frac{\bar A_5}{\eps},
\]
where the coefficients $\bar A_2, \bar A_3,  \bar A_5$ satisfy the linear system
\be\label{lin-syst-A_i}
\ba
\bar A_2 +\bar A_3 =0,\\
\frac{ik}{\bar\Lambda_2}\bar A_2 + \frac{ik}{\bar\Lambda_3}\bar A_3=\mathfrak w,\\
\frac{1}{i\om} (\bar\Lambda_2 \bar A_2 +\bar\Lambda_3\bar A_3 ) + \frac{1}{i\om + c\bar\Lambda_5^2} \bar\Lambda_5\bar A_5=0,
\ea
\ee
where $c:=\kappa_0/\nu_0>0$.
As a consequence, if we isolate the $\eps^2$ boundary layer part, namely
\[
\cU_{BL,\eps^2}:= \sum_{j\in \{2,3\}} a_j X_{\lambda_j} \exp(i(kx-\omega t) - \lambda_j y),
\]
then
\[\ba
\| u_{BL,\eps^2}\|_{\infty}= O(\eps^{-2} \|(\mathfrak u, \mathfrak w, \mathfrak b)\|_{\infty} ),\\
\| w_{BL,\eps^2}\|_{\infty}= O(\|(\mathfrak u, \mathfrak w, \mathfrak b)\|_{\infty} ),\\
\| b_{BL,\eps^2}\|_{\infty}= O(\eps^{-2}  \|(\mathfrak u, \mathfrak w, \mathfrak b)\|_{\infty} )
\ea
\]
while
\be\label{traces}\ba
u_{BL,\eps^2|y=0}=o(\eps^{-2} ),\\
 w_{BL,\eps^2|y=0}= \mathfrak w \exp(i(kx-\om t)) + O(\eps^{2} ),\\
b_{BL,\eps^2|y=0}=o(\eps^{-2} ).\ea
\ee

\begin{Rmk}
	Note that the boundary condition on the normal component of the velocity is mostly handled by the boundary layer of size $\eps^2$, which is not the case for the tangential component. There is indeed an intricate relationship between the $\eps^2$ and the $\eps^3$ boundary layers, and the boundary term $\mathfrak u$ is in fact distributed among the two boundary layers. 
	
\end{Rmk}

An (almost) exact solution to the linear system (\ref{lin-system})-(\ref{original_BCs}) in the case $c_0^{-1} \leq |k| \leq c_0$, $|\zeta| \ll \eps^{3/2}$ is provided by the sum of incident wave and the boundary layers, i.e.

\be\label{Linear_solution_zero_trace}
\cU^0= \mathcal{U}_{inc}+ \mathcal U_{BL}+c.c.
\ee
where the incident wave is $\mathcal{U}_{\text{inc}}=(\mathfrak u , \, \mathfrak w, \, \mathfrak b) \exp(i(kx+my-\omega t))$, and the boundary layer is  
\[
 \mathcal U_{BL}=-\cB_{\omega, k}^c (\mathfrak u , \, \mathfrak w, \, im \mathfrak b) .
\]
Note that the boundary conditions are satisfied without any error, since the trace of the incident wave is balanced by the boundary layer. The only source of error comes from the action of the diffusion operator on the incident wave.

\subsection{Definition and size of the linear solution}

Since  we wish to work with finite energy solutions,  we will  consider wave packets rather than plane waves, i.e. functions of the form
\be\label{incident_wave_packet}
\Winc:=\int_{\R^2} \widehat A(k,m)  X_{k,m}\exp(i(kx-\omega_{k,m} t + my))\:dk\:dm,
\ee
where $\widehat{A} \in \mathcal S(\R^2)$, and $X_{k,m}:=\begin{pmatrix}
1\\ - \frac{k}{m}\\ \frac{i(k\cos \gamma - m \sin \gamma)}{m \omega_{k,m}}
\end{pmatrix}$.
The boundary layer is then
\be\label{BL-wave_packet}
 \mathcal W_{BL}^0=- \int_{\R^2} \widehat A(k,m) \cB_{\omega,k}  \left(1,- \frac{k}{m}, \frac{-(k\cos \gamma - m \sin \gamma)}{\omega_{k,m}}\right)\:dk\:dm ,
\ee
where $\cB_{\omega,k} $ can be $\cB_{\omega,k}^c$, $\cB_{\omega,k}^{nc,BL}$, or $\cB_{\omega,k}^{no}$, depending on the regime under consideration.  

In order that $\Winc$ is an incident wave packet, we also need that the group velocity is oriented towards the slope. Since the time frequency is given by
$$\omega_{k, m}= \pm {{k\cos \gamma - m \sin \gamma} \over {\sqrt{k^2+m^2}}},$$
then the group velocity reads
$$v_g=\nabla_{k,m} \omega_{k,m}=\pm \frac{m \cos \gamma + k \sin \gamma}{(k^2+m^2)^\frac{3}{2}}\begin{pmatrix}
m \\
-k  
\end{pmatrix}.$$
We thus choose the sign of $\omega_{k,m}$ in such a way that
$$v_g \cdot e_y = \mp \frac{k(m \cos \gamma + k \sin \gamma)}{(k^2+m^2)^\frac{3}{2}}<0,$$
where $e_y$ is the unit vector along the normal axis $y$.\\
In order to focus on the critical case, we will choose functions $\widehat{A}$ which are supported close to $\pm (k_0, m_0)$, where $(k_0, m_0)$ are such that
\[
\frac{(k_0 \cos \gamma - m_0 \sin \gamma)^2}{k_0^2 + m_0^2}= \sin^2 \gamma.
\]
More precisely, we will take
\be\label{def:A}
\widehat A(k,m):= \frac{1}{\eps^2} \chi\left(\frac{k-k_0}{\eps^2}\right)\chi\left( \frac{m-m_0}{\eps^2}\right) +  \frac{1}{\eps^2} \chi\left(\frac{k+k_0}{\eps^2}\right)\chi\left( \frac{m+m_0}{\eps^2}\right) ,
\ee
where $\chi\in \mathcal C^\infty_0(\R)$ is a real valued function. Note that in this case, for all $(k,m)\in \mathrm{Supp } \widehat A$, we have $\zeta_{k,m}:=\omega_{k,m}^2 - \sin^2 \gamma= O(\eps^2)$.

The previous analysis shows that we can decompose $ \mathcal W_{BL}^0$ as $\WBLII^0 + \WBLIII^0$, where
\be \label{def:linear-BL}
\ba
\WBLII^0:=- \int_{\R^2} \widehat A(k,m)\sum_{j \in \{2, 3\}} a_j \exp \left( i(kx-\omega_{k,m} t) - \lambda_j y\right) \begin{pmatrix}
U_{\lambda_j}\\ W_{\lambda_j}\\ B_{\lambda_j}\end{pmatrix} \:dk\:dm ,\\
\WBLIII^0:=- \int_{\R^2} \widehat A(k,m) a_5 \exp \left( i(kx-\omega_{k,m} t) - \lambda_5 y\right) \begin{pmatrix}
U_{\lambda_5}\\ W_{\lambda_5}\\ B_{\lambda_5}\end{pmatrix}\:dk\:dm .
\ea
\ee

{\it In the rest of the paper, we set}
\be\label{def:W0}
\cW^0:= \Winc + \WBLII^0 + \WBLIII^0.\ee

We can then estimate the sizes of $\Winc, \WBLII^0$ and $\WBLIII^0$ in $L^\infty$ and $L^2$:

\begin{lem}\label{lem:size-W0}
Assume that $\widehat{A}$ is given by \eqref{def:A}, where $\chi\in \mathcal C^\infty_0(\R)$ is a real valued function. 
We then have the following properties, for all $t\in \R$:
\begin{itemize}
\item Size of the incident wave packet:
\[
\|\Winc\|_{L^2(R^2_+)}= O(1),\quad \|\Winc\|_{L^\infty(R^2_+)}= O(\eps^2);
\]

\item Size of the $\eps^2$ boundary layer:
\[
\|\WBLII^0\|_{L^2(R^2_+)}= O(1),\quad \|\WBLII^0\|_{L^\infty(R^2_+)}= O(1);
\]
 { Note that the normal component of the velocity  is much smaller (by a factor $O(\eps^2)$).}

\item Size of the $\eps^3$ boundary layer:
\[
\|\WBLIII^0\|_{L^2(R^2_+)}= O(\eps^{3/2}),\quad \|\WBLIII^0\|_{L^\infty(R^2_+)}= O(\eps).
\]
 { Note that the normal component of the velocity  is much smaller (by a factor $O(\eps^3)$).}

\end{itemize}

As a consequence, $\cW^0$ is of size $O(1)$ in $L^\infty$ and $L^2$.

\end{lem}

\begin{proof}
$\bullet$ We start with the incident wave packet. Note that definition \eqref{incident_wave_packet} allows us to extend the definition of $\Winc$ to the whole space $\R^2$. According to the Plancherel formula, we have
\[
\| \Winc\|_{L^2(\R^2_+)}\leq \| \Winc\|_{L^2(\R^2)}= \| \widehat A(k,m) X_{k,m}\|_{L^2(\R^2_{k,m})} \lesssim \| \widehat A\|_{L^2} \lesssim 1
\]
by our choice of scaling. Furthermore,
\[
\| \Winc\|_{L^\infty} \lesssim \| \widehat{A} \|_{L^1(\R^2)} \lesssim \eps^2. 
\]

$\bullet$ We now turn towards the estimates on $\WBLII^0$. We start with the $L^\infty$ estimate, which is simpler. We recall that $a_2, a_3= O(\eps^{-2})$, and that $ \begin{pmatrix}
U_{\lambda_j}\\ W_{\lambda_j}\\ B_{\lambda_j}\end{pmatrix}=O(1)$ for all $k,m$ in the support of $\widehat{A}$. We deduce that
\[
\| \WBLII^0\|_{L^\infty( \R^2_+)}\lesssim \eps^{-2}  \| \widehat{A} \|_{L^1(\R^2)} \lesssim 1.
\]

The $L^2$ norm is slightly more involved. Using the Fubini theorem, we have
\[
\WBLII^0(t,x,y)=-\int_{\R}e^{ikx} \left(\int_{\R }\sum_{j\in \{2,3\}}\widehat{A} (k,m )a_j e^{-i\omega t -\lambda_jy }\begin{pmatrix}
U_{\lambda_j}\\ W_{\lambda_j}\\ B_{\lambda_j}\end{pmatrix} dm\right) dk.
\]
Note that if $(k,m)\in \mathrm{Supp } \widehat{A}$, there exist constants $C,c>0$ (independent of $\eps$) such that for $j=2,3$,
\[
|a_j|\leq \frac{C}{\eps^2},\quad \left\| \begin{pmatrix}
U_{\lambda_j}\\ W_{\lambda_j}\\ B_{\lambda_j}\end{pmatrix} \right\|\leq C,\quad \Re(\lambda_j )\geq \frac{c}{\eps^2}.
\]

Therefore, using the Plancherel theorem,
\begin{eqnarray*}
\| \WBLII^0(t,y)\|_{L^2_x}^2&=&
\int_{\R}\left| \int_{\R }\sum_{j\in \{2,3\}} \widehat{A} (k,m )
a_j e^{-i\omega t -\lambda_j y}\begin{pmatrix}	U_{\lambda_j}\\ W_{\lambda_j}\\ B_{\lambda_j}\end{pmatrix} 
dm \right|^2dk \\
&\leq & \frac{C}{\eps^8} \int_{\R}\left| \int_{\R } \sum_{\pm} \exp\left(-\frac{cy}{\eps^2}\right) \chi\left(\frac{k\pm k_0}{\eps^2}\right)\chi\left(\frac{m\pm m_0}{\eps^2}\right)dm \right|^2dk\\
&\leq & \frac{C}{\eps^4} \exp\left(-\frac{2cy}{\eps^2}\right)\sum_{\pm}\int_{\R} \chi^2 \left(\frac{k\pm k_0}{\eps^2}\right)\:dk\\
&\leq &  \frac{C}{\eps^2} \exp\left(-\frac{2cy}{\eps^2}\right).
\end{eqnarray*}
Integrating this inequality with respect to $y$, we obtain $\| \WBLII^0(t,y)\|_{L^2_{x,y}}\lesssim 1$.

$\bullet$ The $\eps^3$ boundary layer is treated with similar arguments. We recall that $a_5=O(\eps^{-1})$, and that there exists a constant $c$ such that $\lambda_5\gtrsim c \eps^{-3}$ for all $(k,m)\in \mathrm{Supp } A$. Therefore
\[
\| \WBLIII^0\|_{L^\infty( \R^2_+)}\lesssim \eps^{-1}  \| \widehat{A} \|_{L^1(\R^2)} \lesssim \eps,
\]
and
\[
\| \WBLIII^0(t,y)\|_{L^2_x}^2\leq  \frac{C}{\eps^2}\exp\left(-\frac{2cy}{\eps^3}\right)\sum_{\pm}\int_{\R} \chi^2 \left(\frac{k\pm k_0}{\eps^2}\right)\:dk \leq C \exp\left(-\frac{2cy}{\eps^3}\right).
\]
The result follows.
\end{proof}

\begin{Rmk}
\begin{itemize}
	\item	The intuition behind the sizes of the $\eps^2$ and $\eps^3$ boundary layers is the following. Because of the wave packet, both boundary layers are localized in $x$ in a band of width $\eps^{-2}$. As a consequence, the $\eps^2$ (resp. $\eps^3$) boundary layer is localized in a region of size $\eps^{-2} \times \eps^{2}=1$ (resp. $\eps^{-2} \times \eps^3=\eps$). Hence the $L^2$ and $L^\infty$ sizes of the $\eps^2$ boundary layer are of the same order, while the $L^2$ size of the $\eps^3$ boundary layer is
	\[
	\|\WBLIII^0\|_{L^2}\lesssim \|\WBLIII^0\|_{L^\infty} \times (\eps^{-2} \times \eps^3)^{1/2} \lesssim \eps \times \eps^{1/2}= \eps^{3/2}.
	\]

\item The same bounds hold for all derivatives with respect to $x$. Concerning the $y$ derivatives, the same bounds hold for the incident wave packet, but a power $\eps^j$ is lost with each derivation for the $\eps^j$ boundary layer, $j=2,3$.

\end{itemize}

	\end{Rmk}


\section{Weakly nonlinear system}
\label{sec:NL-BL}
This section aims at providing a description of our approximate solution when the weak nonlinearity comes into play. As anticipated in the introduction, the nonlinear solution can be formally seen as a small perturbation of the linear one. 

 {More precisely, we will start from the (almost) exact solution to the linear system $\cW^0$, and we will construct an approximate  solution to the weakly nonlinear system in the form 
$$ \cW_{app} = \cW^0 +\cW^1$$
where the corrector $\cW^1$ compensates the self-interactions of $\cW^0$. In other words, the nonlinear term will be treated as a source term for the linear equation.}
One difficulty  here is to deal with nonlinear interactions in the context of the wave packet approximation, which is an infinite superposition of plane waves.

In the case of near-critical reflection with the scalings suggested by Dauxois and Young \cite{DY}, where $\delta$ is the parameter of the weak nonlinearity and the viscosity is of order $O(\varepsilon^6)$, the weakly nonlinear system is the one considered in (\ref{original_system}).
Let us rewrite this system in the following more compact way. Let $\bP$ be the orthogonal projection onto the divergence free vector fields in $L^2(\R_+)^3$, i.e. fields of the form $\cW=(u,w,b)\in L^2(\R^2_+)^3$ with $\partial_xu+\partial_y w=0$  . For $\cW,\, \cW'$ also belonging to $H^1(\R_+)^3$, define
\[
\cL_\eps \begin{pmatrix} u\\w\\b \end{pmatrix}
=\bP \begin{pmatrix}
-b \sin \gamma\\ - b \cos \gamma \\ u \sin \gamma + w \cos \gamma
\end{pmatrix} - \begin{pmatrix}
\nu_0 \varepsilon^6 \Delta u\\ \nu_0 \varepsilon^6 \Delta w\\ \kappa_0 \varepsilon^6 \Delta b
\end{pmatrix}
\]
and
\be\label{def:Q}
Q(\cW, \cW')= \bP( (u\d_x + w\d_y ) \cW').
\ee
Then system \eqref{original_system} becomes
\be\label{eq:NL-compact}
\d_t \cW + \cL_\eps \cW + \delta Q(\cW, \cW)=0.
\ee

The main result of this section is the following.
\begin{Prop}\label{prop:W1}
	Let $\cW^0$ be given by \eqref{def:W0}. There exists a function $\cW^1$, which satisfies
	\[
	\d_t \cW^1 + \cL_\eps \cW^1= - \delta Q(\cW^0, \cW^0) + r^1,
	\]
	together with the boundary conditions $u^1_{|y=0}=w^1_{|y=0}=0$, $\d_y b^1_{|y=0}=0$, where the remainder $r^1$ satisfies
	\[
	\| r^1\|_{L^2}\leq  C \delta \eps^2,
	\]
	and which can be decomposed as
	\[
	\cW^1= \WSH + \WMF + \WBLII^1 + \WBLIII^1.
	\]
	The terms of the above decomposition satisfy the following properties:
	\begin{itemize}
	\item The $\eps^2$ boundary layer satisfies
	$$
	\|  \WBLII^1 \|_{L^2(\R_+^2)} \lesssim \delta,\quad 	\| \WBLII^1 \|_{L^\infty(\R_+^2)} \lesssim \delta \, ;$$
	Moreover, the normal component of the velocity is smaller by a factor $O(\eps^2)$. Derivatives with respect to $t$ and $x$ are bounded, while each derivative with respect to $y$ has a cost $O(\eps^{-2})$.
	\item The $\eps^3$ boundary layer satisfies
	$$
	\|  \WBLIII^1 \|_{L^2(\R_+^2)} \lesssim \delta \eps^{1/2},\quad 	\| \WBLIII^1 \|_{L^\infty(\R_+^2)} \lesssim \delta \, ;$$
	Moreover, the normal component of the velocity is smaller by a factor $O(\eps^3)$. Derivatives with respect to $t$ and $x$ are bounded, while each derivative with respect to $y$ has a cost $O(\eps^{-3})$.
	\item The second harmonic oscillates around frequencies $\pm 2 \omega_0$, and satisfies
	$$
	\|  \WSH \|_{L^2(\R_+^2)} \lesssim \delta ,\quad 	\| \WSH \|_{L^\infty(\R_+^2)} \lesssim \delta \eps^{2}\, ;$$
	\item The mean flow is much smaller
	$$\| \WMF\|_{H^s(\R^2_+)}\lesssim \delta \eps^2,\quad \| \WMF\|_{W^{s,\infty}(\R^2_+)}\lesssim \delta \eps^3\,,$$for any $s\in \N$.
	For these second harmonic and mean flow contributions, derivatives with respect to $t,x,y$ are bounded.
		\end{itemize}

\end{Prop}
\noindent
Note that the terms do not all have the same size. We need all of them  to get a small remainder $r^1$ and to lift exactly all boundary conditions. However our convergence result will not be precise enough to give a physical meaning to all correctors (in particular to the mean flow). There are several tools that are important in our construction:
\begin{itemize}
\item 	First, since $Q$ is a bilinear operator, we can use  definition \eqref{def:W0} to write $Q(\cW^0, \cW^0)$ as a sum of nine terms. According to Lemma \ref{lem:size-W0}, the largest of these terms is $Q(\WBLII^0, \WBLII^0)$. Hence we will need to carefully write the expression of this quadratic term, and in particular, to outline its dependency in time (i.e. its time-frequency) and in space (frequency in the tangential variable, decay in $y$).

\item By linearity, it is also  clear that every term in $Q(\cW^0, \cW^0)$ (except for $Q( \Winc, \Winc)$)  is a linear superposition of terms of the type
\[
\exp(i (kx-\omega t) -\lambda y),
\]
with $\Re(\lambda)\gg 1$, and $k$ (resp. $\omega$) is in a neighborhood of size $\eps^2$ of $0$ or $2 k_0$ (resp. $0$ or $2 \om_0$). Therefore it will be sufficient to understand the behavior of solutions of
\[
\d_t \cW + \cL_\eps \cW = \exp(i (kx-\omega t) -\lambda y),
\]
with homogeneous boundary conditions. We will perform an almost explicit computation of an approximate solution, relying on an asymptotic expansion of $\cL_\eps$ and on the boundary layer analysis of the previous section.


\end{itemize}

Let us now list the nine possible type of interactions by decreasing size in $L^2$, and give their  typical decay (based on the estimates in  Lemma \ref{lem:size-W0}). 
\begin{table}[H]
\caption{List of quadratic interactions}
\centering
\begin{tabular}{|p{2cm}|p{3cm}|p{2.5cm}|p{2.7cm} |}\hline \\ 
{} & {type of interaction} & {size in $L^2$ } & {typical decay rate  }\\ \\ \hline

{(a1)} & $Q(\WBLII^0, \WBLII^0)$ & $O(1)$ & {$\eps^{-2}$ }\\ \\ \hline

{(a2)} & $Q(\Winc, \WBLII^0)$ & $O(1)$ & {$\eps^{-2}$ }\\ \\ \hline

{(b1)} & $Q(\WBLII^0, \WBLIII^0)$ & $O(\eps^{1/2} )$ & {$\eps^{-3}$ }\\ \\ \hline
{(b2)} & $Q(\Winc, \WBLIII^0)$ & $O(\eps^{1/2} ) $& {$\eps^{-3}$ }\\ \\ \hline

{(b3)} & $Q(\WBLIII^0, \WBLII^0)$ &  {$O(\eps^{3/2})$} & {$\eps^{-3}$ }\\ \\ \hline

{(c1)} & $Q(\WBLII^0, \Winc) $ &  {$O(\eps^2)$} & {$\eps^{-2}$ }\\ \\ \hline

{(c2)} & $Q(\Winc, \Winc)$ & $O(\eps^2)$ & {no decay }\\ \\ \hline

{(c3)} & $Q(\WBLIII^0, \WBLIII^0)$ & $O(\eps^{5/2})$ & {$\eps^{-3}$ }\\ \\ \hline

{(c4)} & $Q(\WBLIII^0, \Winc )$ & $O(\eps^{7/2}) $& {$\eps^{-3}$ }\\ \\ \hline

\end{tabular}
\label{Tab:interactions}
\end{table}

\begin{Rmk}
 Note that boundary layer terms of size $\eps^2$ (type (a)) will give birth to  boundary layer terms of size $\eps^2$  AND of size $\eps^3$  (through the boundary condition).
 { Note also that we will not lift terms of type (c) as they are already of size $O(\eps^2)$ and can be directly incorporated to the remainder $r^1$.}
\end{Rmk}

\subsection{Interactions of type (a)}

In this paragraph, we treat the  nonlinear terms which are of size $O(1)$ in $L^2$, with an exponential decay of rate $O(\eps^{-2})$. More precisely, we will prove the following  result:
\begin{Prop}\label{W1a-prop}
	There exists a corrector   $\cW^1_{(a)} $ which satisfies the following properties:
\begin{itemize}
	\item[(i)] $\cW^1_{(a)}$ is a solution of the evolution equation
	\[
	\d_t \cW^1_{(a)} + \cL_\eps \cW^1_{(a)} = - \delta \times \text{(a)}  + r^1_{(a)},
		\]
	with 
	\[
	\| r^1_{(a)}\|_{L^2} \lesssim \delta \eps^2;
	\]

	\item[(ii)] $\cW^1_{(a)}$ satisfies exactly the homogeneous boundary conditions \eqref{BC};
	
	\item[(iii)]  $\cW^1_{(a)}$ can be decomposed as the sum of a mean flow  $\cW^1_{MF;(a)}$, a second harmonic $\cW^1_{II;(a)}$, an $\eps^2$ boundary layer $\cW^1_{BL,\eps^2; (a)}$, and an $\eps^3$ boundary layer $\cW^1_{BL, \eps^3; (a)}$
	$$\cW^1_{(a)}= \cW^1_{BL,\eps^2; (a)} + \cW^1_{BL, \eps^3; (a)} + \cW^1_{II;(a)} + \cW^1_{MF;(a)}\,.$$
	\begin{itemize}
	\item The $\eps^2$ boundary layer satisfies
	$$
	\|  \cW^1_{BL,\eps^2; (a)} \|_{L^2(\R_+^2)} \lesssim \delta,\quad 	\| \cW^1_{BL,\eps^2; (a)} \|_{L^\infty(\R_+^2)} \lesssim \delta \, ;$$
	Moreover, the normal component of the velocity is smaller by a factor $O(\eps^2)$. Derivatives with respect to $t$ and $x$ are bounded, while each derivative with respect to $y$ has a cost $O(\eps^{-2})$.
	\item The $\eps^3$ boundary layer satisfies
	$$
	\|  \cW^1_{BL,\eps^3; (a)} \|_{L^2(\R_+^2)} \lesssim \delta \eps^{1/2},\quad 	\| \cW^1_{BL,\eps^3; (a)} \|_{L^\infty(\R_+^2)} \lesssim \delta \, ;$$
	Moreover, the normal component of the velocity is smaller by a factor $O(\eps^3)$. Derivatives with respect to $t$ and $x$ are bounded, while each derivative with respect to $y$ has a cost $O(\eps^{-3})$.
	\item The second harmonic satisfies
	$$
	\|  \cW^1_{II;(a)} \|_{L^2(\R_+^2)} \lesssim \delta ,\quad 	\| \cW^1_{II;(a)} \|_{L^\infty(\R_+^2)} \lesssim \delta \eps^{2}\, ;$$
	\item The mean flow is much smaller
	$$\| \cW^1_{MF;(a)}\|_{H^s}\lesssim \delta \eps^2,\quad \| \cW^1_{MF;(a)}\|_{W^{s,\infty}}\lesssim \delta \eps^3\,.$$
	For these second harmonic and mean flow contributions, derivatives with respect to $t,x,y$ are bounded.
		\end{itemize}

\end{itemize}

\end{Prop}

 Our strategy will be the following: we will first find an approximate solution of the equation
\[
\d_t \cW + \cL_\eps \cW = - \delta \times \text{(a)}
\]
without taking into account the boundary condition. We then  lift the remaining trace by using the boundary operator of the previous section.  {Note that our construction is very reminiscent from the   construction of Ekman boundary layers for instance.}

  We start with a computation of $Q(\WBLII^0, \WBLII^0)$ and $Q(\Winc, \WBLII^0)$. Using the formulas of the previous section, we have
\begin{equation}\ba
	(u^0_{BL,\eps^2}\d_x + w^0_{BL,\eps^2}\d_y) \WBLII^0
	 = & \int_{\R^4}\sum_{j,j' \in \{2,3\}} \hat A \hat A' a a'e^{i(k+k') x -i(\omega + \omega')t - (\lambda + \lambda') y} \times   \\ & \times (ik' U -\lambda' W) \begin{pmatrix}
		U'\\W'\\B'
	\end{pmatrix} dk\;dm\;dk'\;dm',
\ea\label{exp:(a1)}
\end{equation}
where we wrote in a condensed manner $\widehat A$ for  $\widehat A(k,m)$ defined in (\ref{def:A}),  $a$ for $a_j(k,m)$, $(U,W,B)$ for $(U_{\lambda_j(k,m)} , W_{\lambda_j(k,m)} , B_{\lambda_j(k,m)} )$, see (\ref{def:linear-BL}),  and used the same convention for $\widehat A ', a', (U',W',B')$ (depending on $j',k',m'$). Similarly,
\begin{equation}\ba
	(u^0_{inc}\d_x + w^0_{inc} \d_y )  \WBLII^0
	&=\int_{\R^4}\sum_{j' \in \{2,3\}} \hat A \hat A' a e^{i(k+k') x -i(\omega + \omega')t - (\lambda' - im ) y} \times \\ & \times (ik' U_{inc} -\lambda ' W_{inc}) \begin{pmatrix}
		U'\\W'\\B'
	\end{pmatrix} dk\;dm\;dk'\;dm',
	\ea
\label{exp:(a2)}
\end{equation}
where  $(U_{inc}, W_{inc}, B_{inc})$ denotes here the Fourier transform of the incident wave, see (\ref{incident_wave_packet}).\\

 Note that, in this representation, the Leray projection is simply a multiplier which does not change the general form of the source term. More precisely, since $\mu$ is such that $\Re(\mu)\gtrsim \eps^{-2}$, it is given at leading order by
{$$
 (I-\bP) \left(e^{il x - \mu y}  \begin{pmatrix}
U'\\W'\\B'
\end{pmatrix} \right) = e^{il x - \mu y} \begin{pmatrix} 0 \\ W' \\0
\end{pmatrix}+O(\eps^2) \,,
$$}
 
In particular the second component of $Q(\WBLII^0, \WBLII^0)+Q(\Winc, \WBLII^0)$ is smaller than the other components by a factor $O(\eps^2)$.

Since $\widehat A$ in (\ref{def:A}) is supported in the vicinity of $\pm (k_0, m_0)$, it follows that in the expression above, for $(k,m)\in \mathrm{Supp } \widehat{A}$, $(k',m')\in \mathrm{Supp } \widehat{A}'$, $(k+k', \omega+ \omega')$ is localized either in a ball of size $\eps^2$ around $(0,0)$, or in a ball of size $\eps^2$ around $\pm (2k_0,2\omega_0)$. Hence, by linear superposition, it is sufficient to consider the equation
\be\label{eq:Leps-elem}
\d_t \cW + \cL_\eps \cW = \delta \exp(i l x -i\alpha t - \mu y)\begin{pmatrix}
U'\\W'\\B'
\end{pmatrix},
\ee
where $\mu$ is such that $\Re(\mu)\gtrsim \eps^{-2}$, $U',B'=O(1)$, $W'=O(\eps^2)$, $l$ is close to $0, \, \pm 2 k_0$, $\alpha$ is close to $0, \, \pm 2 \omega_0$. 


\medskip
\noindent
\underbar{ Step 1 : approximation of the linear operator.}
 Let us now compute an approximate operator for   $\cL_\eps$ in this scaling. Since the operator $\cL_\eps$ has constant coefficients, we look for a solution of \eqref{eq:Leps-elem} of the form $\cW=  \delta \exp(ilx -i\alpha t - \mu y) X$, for some vector $X\in \C^3$ to be determined. Because of the divergence free condition, we expect the second component of $X$ to be $O(\eps^2)$. Therefore, looking at the equation for the second component and taking into account the exponential decay, we infer that the pressure is also $O(\eps^2)$, and it is legitimate to keep only the equation on the $u$ and $b$ components of $\cW$  {(which will be checked  a posteriori anyway)}. The equation then becomes
\be\label{eq:cV}
\d_t \cV+ L \cV =  \delta \exp(i l x -i\alpha t - \mu y)\cV' \ee
where $\cV:\R_+\times \R^2_+ \to \bC^2$, $\cV' = (U', B')^T$, and \begin{equation}
\label{L-operator}
L=\sin \gamma \left(\begin{array}{cc}
0 & -1 \\
1 & 0
\end{array}\right).
\end{equation}
It is well-known that the eigenvalues of $L$ are $\pm i\sin \gamma$, and that the corresponding normalized eigenvectors are $V_\pm=\frac{1}{\sqrt{2}} (1, \, \mp i)$.
We therefore have that 
$$ \exp(Lt)=\exp(i\sin \gamma t) \Pi_+ + \exp(-i\sin \gamma t) {\Pi}_-$$
with 
$$\Pi_\pm=\frac{1}{2} \left(\begin{array}{cc} 1 & \pm i \\
	\mp i & 1
	\end{array}\right)\,.$$
We then have the following lemma~:

 {
\begin{lem}\label{V-lem}
There exists a  solution to (\ref{eq:cV}) which  is given explicitly by
\begin{eqnarray*}
	\cV(t,x,y)
	&=& \delta e^{ilx-\mu y-i\alpha t}\sum_{\pm }\frac{1}{-i\alpha \pm i \sin \gamma} \Pi_\pm \cV'.
\end{eqnarray*}
In particular, it has the same decay and the same size as $\cV'$.
\end{lem}}

\begin{proof}
Conjugating the equation (\ref{eq:cV}) by $\exp(Lt)$, we get
$$\begin{aligned}\frac{d}{dt} \left( \sum_\pm e^{\pm i \sin \gamma t} \Pi_\pm \cV' \right)
&=& \delta e^{ilx-\mu y}\sum_{\pm } e^{-i\alpha t \pm i \sin \gamma t}  \Pi_\pm \cV'.
\end{aligned}
$$
Now, by assumption, {recalling that $\zeta=\omega_0^2-\sin^2\gamma \approx O(\varepsilon^2)$ and that $\omega+\omega'$ in formulas (\ref{exp:(a1)})-(\ref{exp:(a2)}) is localized in a neighborhood of $0, \pm 2\omega_0$, then  $$| -\alpha \pm \sin \gamma| \geq  \omega_0/2>0$$ for $\eps$ small enough}. Therefore the right-hand side does not generate any secular growth, and we have
\begin{eqnarray*}
	\cV(t,x,y)&=&\delta e^{ilx-\mu y}\sum_{\pm }\frac{1}{-i\alpha \pm i \sin \gamma}e^{-i\alpha t } \Pi_\pm \cV'.
\end{eqnarray*}
which is  the expected formula.
\end{proof}

We then define $\cV^1_{BL,\eps^2; (a1)}, \cV^1_{BL,\eps^2; (a1)}$ by superposition~:
\begin{eqnarray*}
	\cV^1_{BL,\eps^2; (a1)}&:=&-\delta\int_{\R^4}\sum_{j,j' \in \{2,3\}}\sum_{\pm} \hat A \hat A' a a'e^{i(k+k') x -i(\omega + \omega')t - (\lambda + \lambda') y} \times\\&&\qquad \times \frac{k' U +i\lambda ' W}{-(\omega + \omega') \pm \sin \gamma} \Pi_{\pm}  \begin{pmatrix}
U'\\B'
\end{pmatrix}dk\,dk'\,dm\,dm'\\
\cV^1_{BL,\eps^2; (a2)}&:=&-\delta \int_{\R^4}\sum_{j' \in \{2, 3\}} \hat A \hat A' a e^{i(k+k') x -i(\omega + \omega't) - (\lambda' - im') y} \times\\&&\qquad \times \frac{k' U_{inc} +i \lambda' W_{inc}}{-(\omega + \omega') \pm \sin \gamma} \Pi_{\pm}  \begin{pmatrix}
U'\\B'
\end{pmatrix}dk\,dk'\,dm\,dm,	\end{eqnarray*}
where we recall that in the notation $\cV^1_{BL,\eps^2; (aj)}$, the subscript $(aj$) refers to the classification of quadratic interactions in Table \ref{Tab:interactions}.
By Lemma \ref{lem:size-W0}, it is clear that these quantities are $\eps^2$ boundary layer terms, whose size in $L^2$ and in $L^\infty$ are $O(\delta)$.

\medskip
\noindent
\underbar{ Step 2 : restoring the divergence-free condition.}
 We then define  $\cW_{BL,\eps^2;(a1)}^1$ (resp. $\cW_{BL,\eps^2;(a2)}^1$) by adding the normal velocity of $\cW^1_{BL,\eps^2; (a1)}$ (resp. $\cW^1_{BL,\eps^2; (a2)})$ obtained   simply by  integrating the divergence free condition
 \begin{eqnarray*}
	W^1_{BL,\eps^2; (a1)}
&:=&-\frac12\delta\int_{\R^4}\sum_{j,j' \in \{2,3\}}\sum_{\pm} \hat A \hat A' a a' \frac{i(k+k')}{\lambda+\lambda'}e^{i(k+k') x -i(\omega + \omega')t - (\lambda + \lambda') y} \times\\&&\qquad \times \frac{k' U +i\lambda ' W}{-(\omega + \omega') \pm \sin \gamma}(U'\pm i B') dk\,dk'\,dm\,dm'\\
W^1_{BL,\eps^2; (a2)}&:=&-\frac12\delta \int_{\R^4}\sum_{j' \in \{2,3\}} \hat A \hat A' a' \frac{i(k+k')}{\lambda'-im} e^{i(k+k') x -i(\omega + \omega')t - (\lambda' - im) y} \times\\&&\qquad \times \frac{k' U_{inc} +i\lambda ' W_{inc}}{-(\omega + \omega') \pm \sin \gamma} (U'\pm i B' )dk\,dk'\,dm\,dm	\end{eqnarray*}
The sizes of $W^1_{BL,\eps^2; (a1)}, W^1_{BL,\eps^2; (a2)}$ in $L^2$ and in $L^\infty$ are $O(\delta \eps^2)$, since $\lambda, \lambda' \approx O(\varepsilon^{-2})$.

Note that $\cW_{BL,\eps^2;(a1)}^1$ and  $\cW_{BL,\eps^2;(a2)}^1$ are superpositions of vectors of  respective form
$$\begin{pmatrix}
	1\\ \frac{i(k+k')}{\lambda + \lambda'} \\ \pm i 
\end{pmatrix}, \qquad \begin{pmatrix}
	1\\ \frac{i(k+k')}{-im + \lambda'} \\ \pm i 
\end{pmatrix}\,.$$

By construction, $\sum_{j=1,2} \cW_{BL,\eps^2;(aj)}^1 $ is divergence free, and it can be easily checked that it is a solution of
\[\ba
\d_t \sum _{j=1,2} \cW_{BL,\eps^2;(aj)}^1 + \cL_\eps \sum _{j=1,2} \cW_{BL,\eps^2;(aj)}^1\\ = - \delta Q(\WBLII^0, \WBLII^0)-\delta Q(\Winc, \WBLII^0) +r^1_{(a),L},
\ea\]
where the remainder $r^1_{(a),L}$ (coming from the terms which have been neglected in $\cL_\eps$, i.e. viscous terms, terms involving the normal component $W^1_{BL,\eps^2; (aj)}$ and correctors of order $O(\eps^2)$ in the Leray projection) satisfies 
$$\|r^1_{(a),L}\|_{L^2}= O(\delta \eps^2).$$

\medskip
\noindent
\underbar{ Step 3 : Lifting the boundary conditions.} Now, $\cW_{BL,\eps^2;(a)}^1 $ has a  non-zero trace on the boundary, which must be lifted.
In order to do so, we use the boundary operators of the previous section. Let us recall that $\cW_{BL,\eps^2;(a)}^1 $ oscillates around the frequencies $0$ and $\pm 2 \omega_0$, so that we are never in the critical case. More precisely, we will use either the construction in the case $|\zeta| \gtrsim 1$ (this corresponds to the frequencies $\pm 2\omega_0$, and will give rise to the second harmonic), or the one in the case $|k|\lesssim \nu^{1/3}=\eps^2, |\omega | \lesssim \eps^2$ (this corresponds to the frequencies close to zero, and will give rise to the mean flow). 

Before doing the explicit computation, we need to identify the parts of $\cW_{BL,\eps^2;(a)}^1 $ that will give rise to a second harmonic or to a mean flow. To that end, we write
\begin{eqnarray*}
	\widehat{A} \widehat{A}'&=& \widehat{A}(k,m)\widehat{A}(k',m')\\
	&=& \frac{1}{\eps^4} \sum_{\eta,\eta'\in \{\pm 1\}} \chi\left(\frac{k+ \eta k_0}{\eps^2}, \frac{m+ \eta m_0}{\eps^2}\right) \chi\left(\frac{k'+ \eta' k_0}{\eps^2}, \frac{m'+ \eta' m_0}{\eps^2}\right)\\
	&=& \cA_0 (k,k',m,m') + \cA_{II} (k,k',m,m') ,
\end{eqnarray*}
where
\[
\ba
 \cA_0 (k,k',m,m') = \frac{1}{\eps^4} \sum_{\eta\in \{\pm 1\}} \chi\left(\frac{k+ \eta k_0}{\eps^2}, \frac{m+ \eta m_0}{\eps^2}\right) \chi\left(\frac{k'- \eta k_0}{\eps^2}, \frac{m'- \eta m_0}{\eps^2}\right),\\
  \cA_{II} (k,k',m,m') = \frac{1}{\eps^4}\sum_{\eta\in \{\pm 1\}} \chi\left(\frac{k+ \eta k_0}{\eps^2}, \frac{m+ \eta m_0}{\eps^2}\right) \chi\left(\frac{k'+ \eta k_0}{\eps^2}, \frac{m'+ \eta m_0}{\eps^2}\right).
 \ea
\]
If $(k,k',m,m')\in \mathrm{Supp } \cA_0$, then $(k+k',m+m', \omega+\omega')=O(\eps^2)$, while if $(k,k',m,m')\in \mathrm{Supp } \cA_{II}$, then $(k+k',m+m', \omega+\omega')=\pm 2 (k_0,m_0, \omega_0) + O(\eps^2)$.

Using the expressions above, we infer that
\begin{eqnarray*}
	&& \begin{pmatrix}
	 U^1_{BL, \eps^2; (a)}\\  W^1_{BL, \eps^2; (a)}\\ \d_ y  B^1_{BL, \eps^2; (a)}
	\end{pmatrix}_{|y=0}\\
 &=& -\delta \int_{\R^4}\left(\cA_0 + \cA_{II} \right)  e^{i(k+k') x - i (\om + \om')t} \begin{pmatrix}
 	\mathfrak u_{(a)} \\ \mathfrak w_{(a)} \\ \mathfrak b_{(a)}
 \end{pmatrix}dk \:dk'\:dm\: dm',
\end{eqnarray*}
where
\begin{eqnarray*}
\begin{pmatrix}
	\mathfrak u_{(a)} \\ \mathfrak w_{(a)} \\ \mathfrak b_{(a)}
\end{pmatrix}&=& \frac{1}{2}\sum_{j,j' \in \{2,3\}} \sum_{\pm}a a' \frac{k' U +i \lambda'  W}{-(\omega + \omega') \pm \sin \gamma} (U' \pm i B') \begin{pmatrix}
	1\\ \frac{i(k+k')}{\lambda + \lambda'} \\ \pm i (\lambda + \lambda')
\end{pmatrix}\\
& &+ \frac{1}{2}\sum_{j' \in \{2,3\}} \sum_{\pm}a'  \frac{k' U_{inc} + i\lambda ' W_{inc}}{-(\omega + \omega') \pm \sin \gamma} (U' \pm i B') \begin{pmatrix}
	1\\ \frac{i(k+k')}{\lambda'-im  } \\ \pm i (\lambda'-im) 
\end{pmatrix}\\&=& \sum_{j,j' \in\{2,3\}} \sum_{\pm}O(\eps^{-4}) \begin{pmatrix}
	1\\ \frac{i(k+k')}{\lambda + \lambda'} \\ O(\eps^{-2})
\end{pmatrix} + \sum_{j' \in \{2,3\}} \sum_{\pm}O(\eps^{-4}) \begin{pmatrix}
	1\\ \frac{i(k+k')}{\lambda'-im  } \\ O(\eps^{-2})
\end{pmatrix}.
\end{eqnarray*}
We recall that $a,a', \lambda, U, U'$, etc. are condensed notations for $a_{j}, a_{j'}$, $\lambda_j (k,m,\omega)$,  $U_{\lambda}, U_{\lambda'}$ respectively, see (\ref{def:linear-BL}), while we refer to (\ref{incident_wave_packet}) for the notations $U_{inc}, W_{inc}$.

We now lift the contributions of $\cA_0$ and $\cA_{II}$ using the results of the preceding section. {Note that the time frequency of $\cA_{II}$ is $\pm 2 \omega_0$, therefore it fits the framework of the non-critical case discussed in Lemma \ref{lem:boundary-op-noncritical}. On the other hand, $\cA_0$ is a non oscillating term in time and we will apply the results of Lemma \ref{lem:boundary-op-non-osc}.}

\medskip
\noindent
$\bullet$  Using	Lemma \ref{lem:boundary-op-noncritical}, we define the boundary layer corrector associated with $\cA_{II}$
\[
\begin{aligned}
 \cW_{BL,\eps^3;(aII)}^1:& = \delta  \int_{\R^4}\cA_{II} (k,k',m,m') \cB^{nc,BL}_{\omega + \omega', k+ k'}[
\mathfrak u_{(a)} , \mathfrak w_{(a)} ,\mathfrak b_{(a)}
]dk\:dk'\:dm\:dm' \\
& = \delta \sum_{j=3,5}  \int_{\R^4}\cA_{II} (k,k',m,m')\alpha_j  e^{i((k+k') x-(\omega+\omega')  t)-\Lambda_j y} dk\:dk'\:dm\:dm'
\end{aligned}
\]
and the second harmonic flow associated with the term $(a)$
\[
\begin{aligned}
\cW^1_{II;(a)}:& = \delta  \int_{\R^4}\cA_{II}(k,k',m,m') \cB^{nc,RW}_{\omega + \omega', k+ k'} [
\mathfrak u_{(a)} , \mathfrak w_{(a)} ,\mathfrak b_{(a)}
]dk\:dk'\:dm\:dm'\\
&=\delta   \int_{\R^4}\cA_{II} (k,k',m,m')\alpha_2  e^{i((k+k') x-(\omega+\omega')  t)-\Lambda_2 y} dk\:dk'\:dm\:dm',
\end{aligned}\]
where $\Lambda_2,\Lambda_3, \Lambda_5$ denote the roots with positive real parts of the determinant of the matrix $A_{\nu, \kappa} (\omega+\omega', k+k', \Lambda_j)$ defined in (\ref{def:matrix-A}), associated to $\omega+\omega', k+k'$.
The coefficients $\alpha_2,\alpha_3,\alpha_5$ satisfy (\ref{lin-syst-aj}), i.e.
$$\ba
\alpha_2 + \alpha_3 + \alpha_5 =\mathfrak u_{(a)}{= O(\varepsilon^{-4})},\\
\frac{1}{\Lambda_2} \alpha_2 + \frac{1}{\Lambda_3} \alpha_3+ \frac{1}{\Lambda_5} \alpha_5={\mathfrak w_{(a)}\over i(k+k') }{= O(\varepsilon^{-2})}  ,\\
\tilde \Lambda_2 \alpha_2+\tilde \Lambda_3 \alpha_3 +\tilde \Lambda_5 \alpha_5=-{ 1  \over \sin \gamma}  \mathfrak b_{(a)}{= O(\varepsilon^{-6})},
\ea
$$
with $\tilde \Lambda_2 \sim \frac{\Lambda_2}{\omega+\omega'} = O(1)$, $ \Lambda_3 , \Lambda_5 = O(\eps^{-3})$ and 
$$\tilde \Lambda_3 \sim{  \Lambda_3 \over i(\omega+\omega')+ \varepsilon^6\kappa_0  \Lambda_3^2},\qquad \tilde \Lambda_5 \sim{  \Lambda_5 \over i (\omega+\omega')+ \varepsilon^6\kappa_0  \Lambda_5^2}\,.$$

We can prove the following result:

 {\begin{lem}
The additional boundary corrector satisfies the following estimate
\[
	\|  \cW_{BL,\eps^3;(aII)}^1 \|_{L^2(\R_+^2)} \lesssim \delta \eps^{1/2},\quad 	\| \cW_{BL,\eps^3;(aII)}^1 \|_{L^\infty(\R_+^2)} \lesssim \delta .
	\]
If $\sin \gamma  > 1/2$, the second harmonic flow $\cW^1_{II;(a)}$ is evanescent far from the boundary.
\[
	\|  \cW^1_{II;(a)} \|_{L^2(\R_+^2)} \lesssim \delta \eps,\quad 	\| \cW^1_{II;(a)} \|_{L^\infty(\R_+^2)} \lesssim \delta \eps^{2} .
	\]
 If $\sin \gamma  <1/2$, the  second harmonic flow $\cW^1_{II;(a)}$ is
 a wave packet propagating in the outer domain, and satisfying the following estimates
	\[
	\|  \cW^1_{II;(a)} \|_{L^2(\R_+^2)} \lesssim \delta ,\quad 	\| \cW^1_{II;(a)} \|_{L^\infty(\R_+^2)} \lesssim \delta \eps^2 .
	\]
\label{lem:second-harmonic-a}
\end{lem}}

\begin{proof}
We need first to estimate the coefficients $\alpha_2,\alpha_3,\alpha_5$, which means that we have to compute a good approximation of $C^{-1}$ for 
\be\label{def:C}C =  \begin{pmatrix}1&1&1\\ \frac1{\Lambda_2} &\frac1{\Lambda_3}&\frac1{\Lambda_5} \\
\tilde \Lambda_2 & \tilde \Lambda_3 &\tilde \Lambda_5 \end{pmatrix} . \ee
The determinant of the matrix $C$ is of the form 
$$\det C = \sum_{i,j} \pm {\tilde \Lambda_i \over  \Lambda_j} \sim {\tilde \Lambda_3-\tilde \Lambda_5 \over \Lambda_2} = O(\eps^{-3})\,.$$
We then obtain the inverse matrix by computing the comatrix. At leading order, we get
\begin{equation}
\label{A-1}
 C^{-1} \sim O(\eps^3)  \begin{pmatrix}\frac{\tilde \Lambda_5}{\Lambda_3} - \frac{\tilde \Lambda_3 }{\Lambda_5} &\tilde \Lambda_3 - \tilde \Lambda_5&\frac1{\Lambda_5} - \frac1{\Lambda_3}\\ -\frac{\tilde \Lambda_5}{\Lambda_2}  &\tilde \Lambda_5&\frac1{\Lambda_2} \\
\frac{\tilde \Lambda_3}{\Lambda_2}  &-\tilde \Lambda_3&-\frac1{\Lambda_2} \end{pmatrix}  
\end{equation}
Combining this formula with the estimates on $\mathfrak u_{(a)} , \mathfrak w_{(a)} ,\mathfrak b_{(a)}$, we obtain that
$$ \alpha_2 = O(\eps^{-2}), \quad \alpha_3,\alpha_5 = O(\eps^{-4})\,.$$

Integrating with respect to $m,m', k-k'$ as already done in the proof of Lemma \ref{lem:size-W0}, we obtain that the Fourier transform with respect to $x$ of  $\cW_{BL,\eps^3;(aII)}^1 $ is $O(\delta \eps^{-2})$ in $L^\infty_{k+k',y}$ with decay rate $O(\eps^{-3})$ in $y$,  and supported in a small domain of size $O(\eps^2)$ in $k+k'$. We  then have the following estimates
	$$\|  \cW_{BL,\eps^3;(aII)}^1 \|_{L^2(\R_+^2)} \lesssim \delta \eps^{1/2}   , \quad  \|  \cW_{BL,\eps^3;(aII)}^1 \|_{L^\infty(\R_+^2)} \lesssim \delta \,.$$

Recall that  $\Lambda_2= O(1)$ has a non negative real part and satisfies approximately
	\be\label{time-freq-cap-lambda2} ((\omega+\omega')^2 - \sin^2 \gamma) \Lambda_2^2 - 2i (k+k')\sin \gamma \cos \gamma \Lambda_2 + (k+k')^2 (\cos^2 \gamma - (\omega+\omega')^2)=0\,.\ee
	Since $k+k' =\pm 2k_0 +O(\eps^2)$ and $\omega+\omega ' = \pm 2\omega_0 +O( \eps^2)$, we have that 
	$$\Lambda_2 =\pm {2ik_0 \sin \gamma \cos \gamma \mp 4k_0 \sqrt{ \sin^2\gamma (3\sin^2\gamma - \cos^2\gamma)} \over 3 \sin^2\gamma} +O(\eps^2):= \Lambda_0 + O(\eps^2) \,. $$
	More precisely, we can expand $\omega+\omega'-2\omega_0 $, and then $\Lambda_2-\Lambda_0 $ as a (linear) function of $m-m_0$ and $k-k_0$.
	\begin{itemize}
	\item If $\sin \gamma  > 1/2$ then $\tan \gamma >1/3$ and $\Lambda_0$ has a nonnegative real part. Integrating with respect to $m,m', k-k'$ as done in the proof of Lemma \ref{lem:size-W0}, we obtain that the Fourier transform with respect to $x$ of  $\cW^1_{II;(a)} $ is $O(\delta )$ in $L^\infty_{k+k',y}$,  and supported in a small domain of size $O(\eps^2)$ in $k+k'$. We  then have the following estimates
	$$\|  \cW^1_{II;(a)} \|_{L^2(\R_+^2)} \lesssim \delta \eps , \quad  \|  \cW^1_{II;(a)} \|_{L^\infty(\R_+^2)} \lesssim \delta \eps^{2}\,.$$
\item If $\sin \gamma  < 1/2$, then $\tan \gamma <1/3$ and $\Lambda_0$ is purely imaginary. Since there is no decay in $y$ at main order, we have to use the integral with respect to $m,m'$ to get some integrability. 
Going back to the expression of $\det A_{\nu, \kappa, \omega + \omega', k+k' }$ in Section \ref{sec:linear_BL}, we find that
 there exist real constants  $\mu_1, \mu_2$ such that 
	$$ \Lambda_2 -\Lambda_0 = i \mu_1 ( m+m'-2m_0)+ i \mu_2 ( k+k'-2k_0) + O(\eps^4)\,.$$
	In other words, up to a change of variable (normal form), $m+m'$ is the Fourier variable of $y$. Integrating with respect to $m-m', k-k'$, we obtain that the Fourier transform with respect to $x,y$ of  $\cW^1_{II;(a)} $ is $O(\delta\eps^{-2} )$ in $L^\infty$,  and supported in a small domain of size $O(\eps^4)$.  We  then have the following estimates
	$$\|  \cW^1_{II;(a)} \|_{L^2(\R_+^2)} \lesssim \delta  , \quad  \|  \cW^1_{II;(a)} \|_{L^\infty(\R_+^2)} \lesssim \delta \eps^{2}\,.$$
	\end{itemize}
\end{proof}

\medskip
\noindent
$\bullet$  The contribution of $\cA_0$ is slightly more complicated, because the boundary operator only lifts two of the components. We first consider the boundary layer part due to $\cA_0$, namely 
\[
\cW^1_{BL,\eps^3; (aMF)}  = \delta  \int_{\R^4}\cA_{0 }(k,k',m,m') \cB^{no,BL}_{\omega + \omega', k+ k'}[
	\mathfrak u_{(a)} ,\mathfrak b_{(a)}
	]dk\:dk'\:dm\:dm'.
	\]

Then as observed in Lemma \ref{lem:boundary-op-non-osc}, there remains a non-zero trace for the $w$ component, namely
 \[\mathfrak W^1_{(a)} := \delta  \int_{\R^4}\cA_{0 }e^{i(k+k') x - i (\om + \om')t}\left[ - \mathfrak w_{(a)} + \sum_{l\in \{3,5\}} \frac{i(k+k')}{{\Lambda}_l } \bar{\alpha}_l\right],\]
where $\bar \alpha_3, \bar \alpha_5$ solve the system
\[\ba
\bar \alpha_3 +  \bar \alpha_5= \mathfrak u_{(a)}=O(\varepsilon^{-4}),\\
\sum_{j\in \{3,5\}}- \Lambda_j \bar \alpha_j\left[\frac{\sin \gamma +\frac{i(k+k')}{\Lambda_j}\cos \gamma}{i(\om +\om') - \kappa ((k+k') ^2- \Lambda_j^2)}\right]=\mathfrak b_{(a)}=O(\varepsilon^{-6}),
\ea
\]

with $ \Lambda_3, \Lambda_5=O(\varepsilon^{-3})$ and $k+k', \omega + \omega'=O(\varepsilon^2)$.

In particular, $\bar \alpha_3, \bar \alpha_5 = O(\eps^{-4})$ and we obtain as previously that
\[
	\|  \cW_{BL,\eps^3;(aMF)}^1 \|_{L^2(\R_+^2)} \lesssim \delta \eps^{1/2} ,\quad 	\| \cW_{BL,\eps^3;(aMF)}^1 \|_{L^\infty(\R_+^2)} \lesssim \delta .
	\]

In addition, we can prove the following result:

\begin{lem}
The boundary term $\mathfrak W^1_{(a)}$ can be lifted  by a mean flow $\cW^1_{MF;(a)}$ with divergence free velocity, satisfying the following estimates:
	$$\| D_{t,x,y}^s \cW^1_{MF;(a)}\|_{L^2(\R_+^2) }\lesssim \delta \eps^2,\quad \| D_{t,x,y}^s \cW^1_{MF;(a)}\|_{L^\infty(\R_+^2)}\lesssim \delta \eps^3\,.$$
	In particular
	$$ \d_t \cW^1_{MF;(a)} + \cL_\eps \cW^1_{MF;(a)} = r^1_{(a), MF}$$
	with
	$$ \| r^1_{(a), MF} \|_{L^2(\R_+^2)} \lesssim \delta \eps^2 \,.$$ 
\label{lem:mean-flow-a}
\end{lem}

 {
\begin{proof}
Using the estimates on $(\mathfrak u_{(a)} , \mathfrak w_{(a)} , \mathfrak b_{(a)})$ and the fact that $k+k' = O(\eps^2)$, we can check that 
$ {\bar \alpha_j / \Lambda_j} = O(\eps^{-1} )$ for $j= 3,5$, and   $\mathfrak w_{(a)} = O(1)$. Integrating with respect to $k,k',m,m'$, we deduce that the boundary term to be lifted 
 $\mathfrak W^1_{(a)}$ satisfies
	\[
	\| \mathfrak W^1_{(a)} \|_{L^2(\R)} \lesssim \delta \eps^3,\quad 	\| \mathfrak W^1_{(a)} \|_{L^\infty(\R)} \lesssim \delta \eps^4.
	\]
Furthermore we can build  a function $G\in L^2\cap L^\infty(\R)$ such that $ \mathfrak W^1_{(a)}=\d_x G$, given by
 \[G := \delta  \int_{\R^4}\cA_{0 }e^{i(k+k') x - i (\om + \om')t}\left[ - {\mathfrak w_{(a)}\over i(k+k')}  + \sum_{l\in \{3,5\}} \frac{1}{\lambda_l } \alpha_l\right].\]
 We then have the following estimates~:
\[
\| G\|_{L^2(\R)} \lesssim \delta \eps,\quad \| G\|_{L^\infty(\R)} \lesssim \delta \eps^2.
\]
We then define
$$\cW^1_{MF;(a)}(x, y)= (- G(x) \eps^2 \theta'(\eps^2 y), \theta(\eps^2 y) \d_xG(x), 0)$$
 for some function $\theta \in \mathcal C^\infty_0(\R)$ such that $\theta\sim 1$ in a neighbourhood of zero, with derivative $\theta'$.
	As a consequence, the first two-components of $\cW^1_{MF;(a)}$ are divergence free, and by definition $\cW^1_{MF;(a)}$ lifts the remaining trace~:
$$\cW^1_{MF;(a)|y=0}= (0, -\mathfrak W^1_{(a)} ,0)\,.$$
The explicit formula for $\cW^1_{MF;(a)}$ leads immediately to
$$\| \cW^1_{MF;(a)}\|_{L^2}\lesssim \delta \eps^2,\quad \| \cW^1_{MF;(a)}\|_{L^\infty}\lesssim \delta \eps^3\,.$$
We conclude the proof by observing that each derivative (with respect to $x,y$) actually improves the bound by a factor $O(\eps^2)$. 
We could also differentiate with respect to time, which would not change the estimates.
\end{proof}}

\medskip
\noindent
\underbar{ Step 4 : consistency of the approximation.}

Define
$$\ba 
\cW^1_{(a)}:=( \cW^1_{BL,\eps^2; (a1)}+\cW^1_{BL,\eps^2; (a2)}   + (\cW^1_{BL, \eps^3; (aII)}+\cW^1_{BL,\eps^3; (aMF)})\\ + \cW^1_{II;(a)} + \cW^1_{MF;(a)}.
\ea$$
It satisfies all the properties stated in Proposition \ref{W1a-prop}, with  the remainder
$$ r^1_{(a)} =  r^1_{(a),L}+ r^1_{(a),MF}\,.$$

\subsection*{Interactions of type (b)}

We will now turn to the  nonlinear terms which are of size $O(\eps^{1/2})$ in $L^2$, with an exponential decay of rate $O(\eps^{-3})$. More precisely, we will prove the following  result:
\begin{Prop}\label{W1b-prop}
	There exists a corrector   $\cW^1_{(b)} $ which satisfies the following properties:
\begin{itemize}
	\item[(i)] $\cW^1_{(b)}$ is a solution of the evolution equation
	\[
	\d_t \cW^1_{(b)} + \cL_\eps \cW^1_{(b)} = - \delta \times \text{(b)}  + r^1_{(b)},
		\]
	with 
	\[
	\| r^1_{(b)}\|_{L^2} \lesssim \delta \eps^2;
	\]

	\item[(ii)] $\cW^1_{(b)}$ satisfies exactly the homogeneous boundary conditions \eqref{BC};
	
	\item[(iii)]  $\cW^1_{(b)}$ can be decomposed as the sum of a mean flow  $\cW^1_{MF;(b)}$, a second harmonic $\cW^1_{II;(b)}$, and an $\eps^3$ boundary layer $\cW^1_{BL, \eps^3; (b)}$
	$$\cW^1_{(b)}= \cW^1_{BL, \eps^3; (b)} + \cW^1_{II;(b)} + \cW^1_{MF;(b)}\,.$$
	\begin{itemize}
	\item The $\eps^3$ boundary layer satisfies
	$$
	\|  \cW^1_{BL,\eps^3; (b)} \|_{L^2(\R_+^2)} \lesssim \delta \eps^{1/2},\quad 	\| \cW^1_{BL,\eps^3; (b)} \|_{L^\infty(\R_+^2)} \lesssim \delta \, .$$
	Moreover, the normal component of the velocity is smaller by a factor $O(\eps^3)$. Derivatives with respect to $t$ and $x$ are bounded, while each derivative with respect to $y$ has a cost $O(\eps^{-3})$.
	\item The second harmonic satisfies
	$$
	\|  \cW^1_{II;(b)} \|_{L^2(\R_+^2)} \lesssim \delta \eps,\quad 	\| \cW^1_{II;(b)} \|_{L^\infty(\R_+^2)} \lesssim \delta \eps^{3}\, ;$$
	\item The mean flow is much smaller
	$$\| \cW^1_{MF;(b)}\|_{H^s}\lesssim \delta \eps^2,\quad \| \cW^1_{MF;(b)}\|_{W^{s,\infty}}\lesssim \delta \eps^3\,.$$
	For these second harmonic and mean flow contributions, derivatives with respect to $t,x,y$ are bounded.
		\end{itemize}

\end{itemize}

\end{Prop}

  We start with a computation of $Q(\WBLII^0, \WBLIII^0)$, $Q(\Winc, \WBLIII^0)$ and $Q(\WBLIII^0, \WBLII^0) $. Using the formulas of the previous section and the leading order approximation of the Leray projector in this scaling, we obtain that 
$$\ba  Q(\WBLII^0, \WBLIII^0)
	& =\int_{\R^4}\sum_{j\in \{2,3\}}Ê \hat A \hat A' a a'_5 e^{i(k+k') x -i(\omega + \omega't) - (\lambda + \lambda_5') y} \times \\
	& \times (ik' U -\lambda_5' W) \begin{pmatrix}
		U_5'\\O(\eps^3) \\B_5'
	\end{pmatrix} dk\;dm\;dk'\;dm',
\ea$$
$$\ba  Q(\Winc, \WBLIII^0)
	& =\int_{\R^4} \hat A \hat A' a'_5 e^{i(k+k') x -i(\omega + \omega't) - (\lambda'_5 - im) y} \times \\
	& \times (ik' U_{inc} -\lambda' _5 W_{inc}) \begin{pmatrix}
		U_5'\\O(\eps^3 ) \\B_5'
	\end{pmatrix} dk\;dm\;dk'\;dm',
\ea $$
 and
 $$\ba  Q(\WBLIII^0, \WBLII^0)&=\int_{\R^4}\sum_{j' \in \{2,3\}}Ê \hat A \hat A' a_5 a'e^{i(k+k') x -i(\omega + \omega')t - (\lambda_5 + \lambda') y} \times\\
 & \times  (ik' U_5 -\lambda'  W_5) \begin{pmatrix}
		U'\\O(\eps^3) \\B'
	\end{pmatrix} dk\;dm\;dk'\;dm',\ea$$
	where we recall that the quadratic term $Q(\cW, \cW')$ has been defined in (\ref{def:Q}) and the notations $\hat{A}, a_j, U, W, U_{inc}, W_{inc}, \Winc, \cW_{BL, \varepsilon^2}^0, \cW_{BL, \varepsilon^3}^0$ have been introduced in (\ref{incident_wave_packet}), (\ref{def:linear-BL}), while $U_5, W_5$ are the tangential and the normal component of the velocity field associated to the boundary layer of decay $\exp(-y/\varepsilon^3)$ in (\ref{def:linear-BL}).
 
We have thus to study the solutions to the equation
\be
\d_t \cW + \cL_\eps \cW = \delta \exp(i l x -i\alpha t - \mu y)\begin{pmatrix}
U'\\ W' \\B'
\end{pmatrix} ,
\ee
where $\mu$ is such that $\Re(\mu)\gtrsim \eps^{-3}$, $U',B'=O(1)$, $W'=O(\eps^3)$, and $l$ (resp. $\alpha$) so located in an $\eps^2$ neighborhood of $0$, $2k_0$, or $-2k_0$ (resp. of $0$, $2\omega_0$, or $-2\omega_0$).

 The strategy will be  very similar to the one used in  the previous paragraph, with one major difference~:
   since the scaling is different,  the linear operator contains viscous dissipation in the $y$ direction at leading order.

\medskip
\noindent
\underbar{ Step 1 : approximation of the linear operator.}
At leading order, the equation then becomes
\be\label{eq:cVnu}
\d_t \cV+ L \cV-   \begin{pmatrix} \nu_0 & 0\\ 0 &\kappa_0\end{pmatrix}(\eps^6 \d_{yy} ) \cV =  \delta \exp(i l x -i\alpha t - \mu y) V' \ee
where $\cV:\R_+\times \R^2_+ \to \bC^2$, and  $V' = (U', B')^T$.

 {
\begin{lem}\label{M-lem}
There exists a matrix $M$ depending on $\alpha, \nu_0, \kappa_0$ and $m = \mu\eps^3$, but uniformly bounded with respect to all these parameters, such that 
\begin{eqnarray*}
	\cV(t,x,y)
	=\delta  e^{ilx-\mu y-i\alpha t} M V'.
\end{eqnarray*}
is a solution to (\ref{eq:cVnu}).
\end{lem}}

\begin{proof}
We look at a solution to (\ref{eq:cVnu}) of the form 
$$ \cV(t,x,y)
	=\delta  e^{ilx-\mu y-i\alpha t}  V  \,.$$
	Plugging this Ansatz in the equation, we obtain the following system
	$$ 
	M^{-1} V := \begin{pmatrix}   -i\alpha -\nu_0 m^2 & -\sin\gamma\\ \sin \gamma & -i\alpha -\kappa_0 m^2 \end{pmatrix} V = V'\,,$$
	where $m$ is an approximate solution with positive real part to
	$$ \nu_0\kappa_0 m^2 +i \omega_0( \nu_0+\kappa_0) = 0\,,$$
	{which is the main order equation for the eigenvalues of order $\nu^{-\frac{1}{2}}=\varepsilon^{-3}$ in the critical case with small diffusivity, see section \ref{sec:linear_BL}.}
	
	The matrix $M$ is therefore defined by  
	$$ M = (\sin^2\gamma -\alpha^2 +i\alpha m^2(\nu_0+\kappa_0) +\nu_0 \kappa_0 m^4)^{-1}  \begin{pmatrix}   i\alpha -\kappa_0 m^2 & -\sin\gamma\\ \sin \gamma & i\alpha -\nu_0 m^2 \end{pmatrix}\,.$$
	We have
	$$ 
	\begin{aligned}
	\text{det}(M^{-1})&= \sin^2\gamma -\alpha^2 +i\alpha m^2(\nu_0+\kappa_0) +\nu_0 \kappa_0 m^4\\
	& \sim \omega_0^2 -\alpha^2+ {(\nu_0+\kappa_0)^2\over \nu_0 \kappa_0}  \omega_0 (\alpha - \omega_0) 
	\end{aligned}
	$$  
	In particular, since $\alpha = \pm 2\omega_0+O(\eps^2)$ or $\alpha = O(\eps^2)$, $\text{det}(M^{-1})$  is bounded away from 0. More precisely, by considering the three different values of $\alpha$, namely $\pm 2\omega_0+O(\eps^2)$, $O(\varepsilon^2)$, we can check that $|\text{det}(M^{-1})| \ge C \omega_0^2$, where the constant value $C$ is independent of $\nu_0, \kappa_0$.
	We therefore obtain a uniform bound on $M$ in all regimes we will consider.
	\end{proof}

We then define $\cV^1_{BL,\eps^3; (b)}$ by superposition~:
\begin{align*}
	\cV^1_{BL,\eps^3; (b1)}&=-\delta \int_{\R^4}\sum_{j \in \{2,3\}}Ê \hat A \hat A' a a'_5 e^{i(k+k') x -i(\omega + \omega't) - (\lambda + \lambda_5') y}\times\\ 
	& \times (ik' U -\lambda_5' W) M\begin{pmatrix}
		U_5'\\B_5'
	\end{pmatrix} dk\;dm\;dk'\;dm',
	\end{align*}
	\begin{align*}
\cV^1_{BL,\eps^3; (b2)}&=-\delta \int_{\R^4} \hat A \hat A' a'_5 e^{i(k+k') x -i(\omega + \omega't) - (\lambda'_5 - im) y}\times \\ & \times (ik' U_{inc} -\lambda' _5 W_{inc}) M\begin{pmatrix}
		U_5'\\B_5'
	\end{pmatrix} dk\;dm\;dk'\;dm',
		\end{align*}
		\begin{align*}
	\cV^1_{BL,\eps^3; (b3)}&=\int_{\R^4}\sum_{j'\in \{2,3\}}Ê \hat A \hat A' a_5 a'e^{i(k+k') x -i(\omega + \omega't) - (\lambda_5 + \lambda') y} \times \\ & \times (ik' U_5 -\lambda'  W_5) M\begin{pmatrix}
		U'\\B'
	\end{pmatrix} dk\;dm\;dk'\;dm'.
		\end{align*}
This quantity is an $\eps^3$ boundary layer term, whose size   is $O(\delta \eps^{1/2})$ in $L^2$  and $O(\delta)$ in $L^\infty$ (as the source term (b)).
We indeed recall that $a,a' = O( \eps^{-2})$ and $a_5, a'_5 = O( \eps^{-1})$.

\medskip
\noindent
\underbar{ Step 2 : restoring the divergence-free condition.}

 We then define the normal velocity of $\cW^1_{BL,\eps^2; (b1)}$, $\cW^1_{BL,\eps^2; (b2)}$ and $\cW^1_{BL,\eps^2; (b3)}$  simply by  integrating the divergence free condition. Denoting by $M_1$ (resp. $M_2$) the first (resp. the second) line of the matrix $M$, we have 
\begin{align*}
W^1_{BL,\eps^3; (b1)}&=-\delta \int_{\R^4}\sum_{j \in \{2,3\}}Ê \hat A \hat A' a a'_5 {i(k+k') \over \lambda + \lambda_5'} e^{i(k+k') x -i(\omega + \omega't) - (\lambda + \lambda_5') y}\times \\ & \times (ik' U -\lambda_5' W) M_1\begin{pmatrix}
		U_5'\\B_5'
	\end{pmatrix} dk\;dm\;dk'\;dm',
	\end{align*}
	\begin{align*}
W^1_{BL,\eps^3; (b2)}&=-\delta \int_{\R^4} \hat A \hat A' a'_5  {i(k+k') \over -im+\lambda_5'} e^{i(k+k') x -i(\omega + \omega't) - (\lambda'_5 - im) y} \times \\ & \times  (ik' U_{inc} -\lambda' _5 W_{inc}) M_1\begin{pmatrix}
		U_5'\\B_5'
	\end{pmatrix} dk\;dm\;dk'\;dm',
		\end{align*}
		\begin{align*}
	W^1_{BL,\eps^3; (b3)}&=\int_{\R^4}\sum_{j'\in \{2,3\}}Ê \hat A \hat A' a_5 a'  {i(k+k') \over \lambda' + \lambda_5}e^{i(k+k') x -i(\omega + \omega't) - (\lambda_5 + \lambda') y} \times \\ & \times (ik' U_5 -\lambda'  W_5) M_1\begin{pmatrix}
		U'\\B'
	\end{pmatrix} dk\;dm\;dk'\;dm'	
	\end{align*}
The size of $W^1_{BL,\eps^3; (bj)}$ is $O(\delta \eps^{7/2})$ in $L^2$  and $O(\delta \eps^3)$ in $L^\infty$.

Note that $\cW_{BL,\eps^3;(b)}^1$ is a superposition of vectors of the form
$$\begin{pmatrix}
	M_1 \cV\\ \frac{i(k+k')}{\lambda + \lambda_5'} M_1 \cV\\ M_2 \cV
\end{pmatrix},\quad\hbox{ and }Ê \begin{pmatrix}
	M_1 \cV\\ \frac{i(k+k')}{-im+ \lambda_5'} M_1 \cV\\ M_2 \cV
\end{pmatrix}\,.$$

By construction, $\sum_{j=1}^3\cW_{BL,\eps^3;(bj)}^1$ is divergence free, and it can be easily checked that it is a solution of
\[
\begin{aligned}
\d_t \left( \sum_{j=1}^3\cW_{BL,\eps^3;(bj)}^1 \right) & + \cL_\eps \left( \sum_{j=1}^3\cW_{BL,\eps^3;(bj)}^1\right)\\
 & = - \delta Q(\WBLII^0, \WBLIII^0)-\delta Q(\Winc, \WBLIII^0)\\
 &\quad \, -\delta Q(\WBLIII^0, \WBLII^0) +r^1_{(b),L},
\end{aligned}
\]
where the remainder $r^1_{(b),L}$ (coming from the terms which have been neglected in $\cL_\eps$, i.e. viscous terms $\d_{xx}$, terms involving $W^1_{BL,\eps^3; (b)}$ and correctors of order $O(\eps^3)$ in the Leray projection) satisfies 
$$\|r^1_{(b),L}\|_{L^2}= O(\delta \eps^{5/2}).$$

\medskip
\noindent
\underbar{ Step 3 : Lifting the boundary conditions.} Now, $\cW_{BL,\eps^3;(bj)}^1 $ has a  non-zero trace on the boundary, which must be lifted.
Note that this trace has exactly the same structure  as in the previous case, namely
\begin{eqnarray*}
	&& \sum_{j=1}^3\begin{pmatrix}
	 U^1_{BL, \eps^3; (bj)}\\  W^1_{BL, \eps^3; (bk)}\\ \d_ y  B^1_{BL, \eps^3; (bj)}
	\end{pmatrix}_{|y=0}\\
 &=& -\delta \int_{\R^4}\left(\cA_0 + \cA_{II} \right)  e^{i(k+k') x - i (\om + \om')t} \begin{pmatrix}
 	\mathfrak u_{(b)} \\ \mathfrak w_{(b)} \\ \mathfrak b_{(b)}
 \end{pmatrix}dk \:dk'\:dm\: dm',
\end{eqnarray*}
where
\begin{eqnarray*}
\begin{pmatrix}
	\mathfrak u_{(b)} \\ \mathfrak w_{(b)} \\ \mathfrak b_{(b)}
\end{pmatrix}&=& \sum_{j \in \{2,3\}}Ê  a a'_5 (ik' U -\lambda_5' W) \begin{pmatrix}
	M_1\\ \frac{i(k+k')}{\lambda + \lambda_5'} M_1\\ -(\lambda + \lambda_5')M_2
\end{pmatrix}(U_5',B_5')^t\\
& &+ \; a'_5 (ik' U_{inc} -\lambda_5' W_{inc}) \begin{pmatrix}
	M_1\\ \frac{i(k+k')}{-im + \lambda_5'} M_1\\ -(-im + \lambda_5')M_2
\end{pmatrix}(U_5',B_5')^t\\\
& & + \sum_{j'\in \{2,3\}}Êa_5 a' (ik' U_5 -\lambda' W_5) \begin{pmatrix}
	M_1\\ \frac{i(k+k')}{\lambda_5 + \lambda'} M_1\\ -(\lambda_5 + \lambda')M_2
\end{pmatrix}(U',B')^t \\
&=& O(\eps^{-4})  \begin{pmatrix}
	M_1\\ \frac{i(k+k')}{\lambda_5 + \lambda'} M_1\\ O(\eps^{-3})
\end{pmatrix}(U',B')^t + 
 O(\eps^{-4})  \begin{pmatrix}
	M_1\\ \frac{i(k+k')}{-im+ \lambda_5'} M_1\\ O(\eps^{-3})
\end{pmatrix}(U_5',B_5')^t \\
& & + O(\eps^{-4})  \begin{pmatrix}
	M_1\\ \frac{i(k+k')}{\lambda_5' + \lambda} M_1\\ O(\eps^{-3})
\end{pmatrix}(U_5',B_5')^t.
\end{eqnarray*}

We now lift the contributions of $\cA_0$ and $\cA_{II}$ exactly as in the previous paragraph.
 Using	Lemma \ref{lem:boundary-op-noncritical}, we first define the boundary layer corrector associated with $\cA_{II}$
\[
\begin{aligned}
\cW_{BL,\eps^3;(bII)}^1 &= \delta  \int_{\R^4}\cA_{II} (k,k',m,m') \cB^{nc,BL}_{\omega + \omega', k+ k'}[
\mathfrak u_{(b)} , \mathfrak w_{(b)} ,\mathfrak b_{(b)}
]dk\:dk'\:dm\:dm' \\
& = \delta \sum_{j=3,5}  \int_{\R^4}\cA_{II} (k,k',m,m')\beta_j  e^{i((k+k') x-(\omega+\omega')  t)-\Lambda_j y} dk\:dk'\:dm\:dm'
\end{aligned}
\]
and the second harmonic flow associated with the term $(b)$
\[
\begin{aligned}
\cW^1_{II;(b)}&:= \delta  \int_{\R^4}\cA_{II}(k,k',m,m') \cB^{nc,RW}_{\omega + \omega', k+ k'} [
\mathfrak u_{(b)} , \mathfrak w_{(b)} ,\mathfrak b_{(b)}
]dk\:dk'\:dm\:dm'\\
&=\delta   \int_{\R^4}\cA_{II} (k,k',m,m')\beta_2  e^{i((k+k') x-(\omega+\omega')  t)- \Lambda_2 y}dk\:dk'\:dm\:dm'.
\end{aligned}\]
where $ \Lambda_2,  \Lambda_3,  \Lambda_5$ denote the roots with positive real parts of the determinant of the matrix $A_{\nu, \kappa} (\omega+\omega', k+k', \Lambda_j)$ defined in (\ref{def:matrix-A}), associated to $\omega+\omega', k+k'$.
The coefficients $\beta_j$ are defined by
$$\begin{pmatrix} \beta_2\\ \beta_3\\\beta_5\end{pmatrix} = C^{-1} \begin{pmatrix} \mathfrak u_{(b)} \\ \mathfrak w_{(b)} \\ -\frac{\mathfrak b_{(b)}}{\sin \gamma}\end{pmatrix} $$
where the matrix $C$ is defined in \eqref{def:C}, and satisfies
\begin{align*}
C^{-1}&\sim O(\eps^3)  \begin{pmatrix}\frac{\tilde \Lambda_5}{\Lambda_3} - \frac{\tilde \Lambda_3 }{\Lambda_5} &\tilde \Lambda_3 - \tilde \Lambda_5&\frac1{\Lambda_5} - \frac1{\Lambda_3}\\ -\frac{\tilde \Lambda_5}{\Lambda_2}  &\tilde \Lambda_5&\frac1{\Lambda_2} \\
\frac{\tilde \Lambda_3}{\Lambda_2}  &-\tilde \Lambda_3&-\frac1{\Lambda_2} \end{pmatrix}\\\\
&=O(\eps^3)  \begin{pmatrix} O(1) & O(\varepsilon^{-3}) & O(\varepsilon^3) \\
O(\varepsilon^{-3}) &  O(\varepsilon^{-3}) &  O(1) \\
O(\varepsilon^{-3}) &  O(\varepsilon^{-3}) & O(1)
\end{pmatrix}.
\end{align*} 
As a consequence,
$$\begin{pmatrix} \beta_2\\ \beta_3\\\beta_5\end{pmatrix} 
= \begin{pmatrix} O(\eps^{-1} ) \\ O(\eps^{-4})  \\ O(\eps^{-4})\end{pmatrix}.$$ 

Note that $\beta_3$ and $\beta_5$ are of the same size as $\alpha_3$ and $\alpha_5$. However, $\beta_2=O(\eps^{-1})$ is a priori smaller than $\alpha_2=O(\eps^{-2}).$ As a consequence, the estimate on $\cW_{BL,\eps^3;(bII)}^1$ is exactly the same as the one on $\cW_{BL,\eps^3;(aII)}^1$, but the estimate on $\cW^1_{II;(b)}$ is smaller than the one on $\cW^1_{II;(a)}$ by a factor $\eps$.

We then define the boundary layer part due to $\cA_0$, namely
\[
\cW^1_{BL,\eps^3; (bMF)}  = \delta  \int_{\R^4}\cA_{0 }(k,k',m,m') \cB^{no,BL}_{\omega + \omega', k+ k'}[
	\mathfrak u_{(b)} ,\mathfrak b_{(b)}
	]dk\:dk'\:dm\:dm'.
	\]
We  finally lift the remaining trace on the $w$ component by the mean flow term
$$\cW^1_{MF;(b)}(x, y)= (- G(x) \eps^2 \theta'(\eps^2 y), \theta(\eps^2 y) \d_xG(x), 0)$$
where  
 \[G := \delta  \int_{\R^4}\cA_{0 }e^{i(k+k') x - i (\om + \om')t}\left[ - {\mathfrak w_{(b)}\over i(k+k')}  + \sum_{l\in \{3,5\}} \frac{1}{\Lambda_l } \bar \beta_l\right],\]
and 
\[
\begin{pmatrix}\bar  \beta_3\\ \bar \beta_5 \end{pmatrix}  = \begin{pmatrix} 1& 1 \\ 
\tilde \Lambda_3 
&\tilde \Lambda_5  \end{pmatrix} ^{-1} \begin{pmatrix}  \mathfrak u_{(b)}\\ -\frac{\mathfrak b_{(b)}}{\sin \gamma}\end{pmatrix} = \begin{pmatrix} O(\eps^{-4}) \\
O(\eps^{-4}) \end{pmatrix} \]
denoting
$$\tilde \Lambda_j = -\Lambda_j \left[\frac{\sin \gamma +\frac{i(k+k')}{\Lambda_j}\cos \gamma}{i(\om +\om') - \kappa ((k+k') ^2-\Lambda_j^2)}\right] = O\left(  \eps^{-3}\right). $$
We then have exactly the same estimates as for the interactions of type (a).

\medskip
\noindent
\underbar{ Step 4 : consistency of the approximation.}

Define
$$\ba
\cW^1_{(b)}:=\Big( \cW^1_{BL,\eps^3; (b1)}+\cW^1_{BL,\eps^3; (b2)}  \cW^1_{BL,\eps^3; (b3)} + \cW^1_{BL, \eps^3; (bII)}+\cW^1_{BL,\eps^3; (bMF)}\Big)\\ + \cW^1_{II;(b)} + \cW^1_{MF;(b)}.
\ea$$
It satisfies all the properties stated in Proposition \ref{W1b-prop}, with  the remainder
$$ r^1_{(b)} =  r^1_{(b),L}+ r^1_{(b),MF}\,.$$


\section{Accuracy of the approximation}
\label{sec:stability}
The aim of this section is to quantify the accuracy of our construction, by providing an $L^2$ estimate of the difference between the approximate solution and the solution to the Cauchy problem for (\ref{eq:NL-compact})-(\ref{BC}) with initial data
$$\mathcal{W}_0=\cW_{app}(t=0)=\mathcal{W}^0(t=0)+\mathcal{W}^1(t=0),$$
where $\mathcal{W}^0, \mathcal{W}^1$ have been defined in (\ref{def:W0}) and Proposition \ref{prop:W1} respectively.

\subsection{The error estimate}
The first step is to estimate the size of the error in the approximation of the evolution equation, coming from terms of different types \,:\\
$\bullet$ the viscosity for the incident wave packet $\Winc$;\\
$\bullet$ the remainder $r^1$ in Proposition \ref{prop:W1} (corresponding to the error coming from $\cW^1$);\\
$\bullet$ the quadratic interactions (c) in Table \ref{Tab:interactions} that have not been corrected;\\
$\bullet$ the interactions between $\cW^0$ and $\cW^1$ and the self-interactions of $\cW^1$.

\begin{Prop}\label{prop:error-estimate}
Consider $\mathcal{W}_{app}=\mathcal{W}^0+\mathcal{W}^1,$
with $\mathcal{W}^0, \mathcal{W}^1$ defined in (\ref{def:W0}) and Proposition \ref{prop:W1} respectively. Then $\mathcal{W}_{app}$ solves 
$$\partial_t \mathcal{W}_{app} + \mathcal{L}_\varepsilon \mathcal{W}_{app}+\delta Q(\mathcal{W}_{app}, \mathcal{W}_{app})=R_{app}$$
and satisfies exactly the boundary conditions (\ref{BC}), where 
$$R_{app}:=r^1+\delta \times (c) + \delta \Bigg(Q(\mathcal{W}^0, \mathcal{W}^1)+Q(\mathcal{W}^1, \mathcal{W}^0)+Q(\mathcal{W}^1, \mathcal{W}^1)\Bigg)-\begin{pmatrix}
\nu_0 \varepsilon^6 \Delta u^0_{inc}\\
\nu_0 \varepsilon^6 \Delta w^0_{inc}\\
\kappa_0 \varepsilon^6 \Delta b^0_{inc}\\
\end{pmatrix}.$$
Moreover,
\begin{itemize}
\item $\|\mathcal{W}_{app}\|_{L^2(\mathbb{R}_+^2)}= O(1), \quad \|\cW_{app}\|_{L^\infty(\R^2_+)} = O(1);$\\
\item  $\|\nabla \mathcal{W}_{app}\|_{L^2(\mathbb{R}_+^2)} = O(\varepsilon^{-2}), \quad \| \nabla \cW_{app}\|_{L^\infty(\R^2_+)} = O(\varepsilon^{-2});$\\
\item $\|R_{app}\|_{L^2(\mathbb{R}_+^2)} = O(\delta \varepsilon^2)+O(\delta^2)+O(\varepsilon^6)$.
\end{itemize}
\end{Prop}
\begin{proof}
The consistency of $\cW_{app}$ can be easily checked by recalling the equations satisfied by $\cW^0, \cW^1$ in Lemma \ref{lem:size-W0} and Proposition \ref{prop:W1} respectively. Again, the sizes of $\cW_{app}$ and $\nabla \cW_{app}$ are a direct consequence of Lemma \ref{lem:size-W0} and Proposition \ref{prop:W1}. Let us estimate the size of the remainder.\\
$\bullet$ From Proposition \ref{prop:W1}, $\|r^1\|_{L^2(\mathbb{R}^2_+)}=O(\delta
\varepsilon^2)$; \\
$\bullet$ From Table \ref{Tab:interactions}, $\|\delta \times (c) \|_{L^2(\mathbb{R}^2_+)}=O(\delta \varepsilon^2)$; \\
$\bullet$ From Lemma \ref{lem:size-W0}, the size of the remainder due to the viscosity applied to the incident wave packet is $O(\varepsilon^6)$; \\
$\bullet$ It remains then to estimate the interactions between $\cW_{app}$ and $\cW^1$:
\begin{align*}
&\delta \| Q(\cW^0, \cW^1)\|_{L^{2}(\R^2_+)} \le \delta \|(u^0 \partial_x + w^0 \partial_y)\cW^1\|_{L^{2}(\R^2_+)}\\
& \le \delta \Bigg(\|u^0\|_{L^\infty} \|\partial_x \cW^1\|_{L^2} + \|w^0\|_{L^\infty} \|\partial_y \cW^1\|_{L^2}\Bigg) = O(\delta^2);\\
\end{align*}
\begin{align*}
& \delta \| Q(\cW^1, \cW^0)\|_{L^{2}(\R^2_+)} \le \delta \|(u^1 \partial_x + w^1 \partial_y)\cW^0\|_{L^{2}(\R^2_+)}\\
& \le \delta \Bigg(\|u^1\|_{L^\infty} \|\partial_x \cW^0\|_{L^2} + \|w^1\|_{L^\infty} \|\partial_y \cW^0\|_{L^2}\Bigg) = O(\delta^2);
\end{align*}
\begin{align*}
\delta \| Q(\cW^1, \cW^1)\|_{L^{2}(\R^2_+)}=O(\delta^3).
\end{align*}
\end{proof}

\begin{Rmk}
	It is natural to wonder whether we could iterate our construction in order to define an approximate solution up to any order. In other words, given $N\in \N$ arbitrary, is it possible to define a solution $\mathcal{W}_{app}$ (different from the solution defined above) such that
	\[
	\partial_t \mathcal{W}_{app} + \mathcal{L}_\varepsilon \mathcal{W}_{app}+\delta Q(\mathcal{W}_{app}, \mathcal{W}_{app})=O(\eps^N)\quad \text{in } L^\infty (\R_+,L^2(\R^2_+))?
	\]
	There are several limitations preventing us to do so:
	\begin{itemize}
	\item First, because of the degeneracy of the root $\lambda_2(k+k',\omega + \omega')$ in the region $(k+k', m+m')\in \mathrm{Supp }\cA_0$, we have chosen to lift the traces $\mathfrak{W}^{1}_{(a)} $ and  $\mathfrak{W}^{1}_{(b)}$  by correctors that have no physical relevance, namely $\cW^1_{MF;(a)}$ and $\cW^1_{MF;(b)}$. In order to reach a better approximation, we need to understand how to lift these traces in a more specific way, which possibly entails tracking down the exact degeneracy of $\lambda_2$.
	
	\item In a similar fashion, in the above construction, we have treated $(I-\mathbb P) ((u^0\d_x + w^0 \d_y) \cW^0)$ as a remainder. If we want to go further in the definition of the approximate solution, we need to understand the precise structure of this remainder and to correct it by additional terms.
	
	\item The next step in the construction of an approximate solution would be to consider the remainder terms stemming from the interactions between $\cW^0$ and $\cW^1$, namely $\delta Q(\cW^0, \cW^1)$ and $\delta Q(\cW^1, \cW^0)$. One could try and apply the same method as above in order to add a corrector $\cW^2$, satisfying
	\[
	\d_t \cW^2 + \cL_\eps \cW_2 = - \delta Q(\cW^0, \cW^1)- \delta Q(\cW^1, \cW^0) + r^2,
	\] 
	with $\| r^2\|_{L^2}\ll \delta^2$.
	
	However, there is a major difference with the case of the self-interactions of $\cW^0$ (i.e. the term $Q(\cW^0, \cW^0)$). Indeed, the term $\cW^0$ oscillates in time around the frequencies $\pm \omega_0$; whence $Q(\cW^0, \cW^0)$ oscillates around the frequencies $\pm 2 \omega_0$ and $0$. Since the resonance frequencies of the rotation operator $L$ are $\pm \omega_0$, the definition of $\cW^1$ does not create any resonance, see Lemma \ref{V-lem}. But the terms $Q(\cW^0, \cW^1)$ and $Q(\cW^1, \cW^0)$ contain the frequencies $\pm 3 \omega_0, \pm \omega_0$, and therefore resonances may occur.
	
	A more careful look at the quadratic terms shows that in fact, the resonant interactions between $\Winc+ \WBLII^0$ and $\WBLII^1$ cancel at main order. This is due to the fact that the eigenvectors $(U_{inc}, B_{inc})$ and $(U_\lambda, B_\lambda)$ are both $(1,\pm i) + O(\eps^2)$ when $|\lambda|\propto \eps^{-2}$. In more abstract terms, this cancellation comes from the Jacobi identity.
	However, this algebraic cancellation no longer holds for the terms involving the $\eps^3$ boundary layer. As a consequence, we cannot lift these interactions with the tools developed above, and new ideas must be introduced if one wants to push the iteration further.

	\end{itemize}

\end{Rmk}

\subsection{The stability inequality}\label{subsec:stability-inequality}
Here we will establish the stability inequality (\ref{energy-inequality}) leading to Theorem \ref{Thm:stability}.\\
Recalling the definitions of $\cL_\varepsilon$ and $Q(\cW, \cW')$ in (\ref{eq:NL-compact}), the approximate solution satisfies the following equation
$$\partial_t \cW_{app} +  \mathbb{P}\mathbf{L} \cW_{app} + \delta \mathbb{P}((u_{app}\partial_x + w_{app} \partial_y) \cW_{app})=\begin{pmatrix}
\nu_0 \varepsilon^6 \Delta u_{app}\\
\nu_0 \varepsilon^6 \Delta w_{app}\\
\kappa_0 \varepsilon^6 \Delta b_{app}\\
\end{pmatrix}+R_{app},$$
where 
$$\mathbf{L}=\begin{pmatrix}
0 & 0 & -\sin \gamma\\
0 & 0 & -\cos \gamma \\
\sin \gamma & \cos \gamma & 0 
\end{pmatrix}.$$

We point out that the strong formulation of the system, which is required in order to get the energy estimate leading to the stability inequality, is actually satisfied by a sequence approximate solutions,  which are smooth by Friedrichs approximation, see the Appendix. Then the energy inequality for the weak solution $\cW$ is obtained by passing to the limit. With a slight abuse of notation, here we omit this passage and we consider directly the equation for $\cW$.\\
We write the equation for the difference $\cW_{app}-\cW$ and take the scalar product against  $\cW_{app}-\cW$. This yields:
\be \ba\label{equation:stability}
\frac{1}{2}\frac{d}{dt}\|\cW_{app}-\cW\|_{L^2(\R^2_+)}^2+ \int_{\mathbb{R}^2} \mathbb{P}\mathbf{L}(\cW_{app}-\cW) (\cW_{app}-\cW)  \, dx \, dy
\\ + \delta \int_{\mathbb{R}^2} \mathbb{P}((u_{app}-u)\partial_x \cW_{app} +(w_{app}-w)\partial_y \cW_{app} ) (\cW_{app}-\cW) \, dx \, dy \\
+ \delta \int_{\mathbb{R}^2} \mathbb{P}((u\partial_x (\cW_{app}-\cW) + w \partial_y (\cW_{app}-\cW) ) (\cW_{app}-\cW) \, dx \, dy \\
= \varepsilon^6 \nu_0 \int_{\mathbb{R}^2} \Delta (u_{app}-u) (u_{app}-u) + \Delta (w_{app}-w) (w_{app}-w)\, dx \, dy \\
+ \varepsilon^6 \kappa_0 \int_{\mathbb{R}^2} \Delta (b_{app}-b) (b_{app}-b) \, dx \, dy
+\int_{\mathbb{R}^2} R_{app}  (\cW_{app}-\cW) \, dx \, dy.
\ea \ee 

We analyse each term separately.
\begin{align*}\bullet \quad
(\mathbb{P}\mathbf{L}(\cW_{app}-\cW), \, \cW_{app}-\cW)_{L^2} = (\mathbf{L}(\cW_{app}-\cW), \, \mathbb{P}(\cW_{app}-\cW))_{L^2} \\
= (\mathbf{L}(\cW_{app}-\cW), \, \cW_{app}-\cW)_{L^2}=0,
\end{align*} 
since $\mathbb{P}$ is symmetric by definition and $\mathbf{L}$ is skew-symmetric.
\begin{align*} \bullet \quad
\delta |(\mathbb{P}((u_{app}-u)\partial_x \cW_{app} +(w_{app}-w)\partial_y \cW_{app},\, \cW_{app}-\cW)_{L^2}| \\
\le \delta \|\mathbb{P}((u_{app}-u)\partial_x \cW_{app} +(w_{app}-w)\partial_y \cW_{app})\|_{L^2} \|\cW_{app}-\cW\|_{L^2}\\
\le \delta \|(u_{app}-u)\partial_x \cW_{app} +(w_{app}-w)\partial_y \cW_{app}\|_{L^2} \|\cW_{app}-\cW\|_{L^2}\\
\le \delta \|\nabla \cW_{app}\|_{L^\infty} \|\cW_{app}-\cW\|_{L^2}^2 \\
\le \delta \varepsilon^{-2}  \|\cW_{app}-\cW\|_{L^2}^2, \\
\end{align*}
where $\mathbb{P}$ is bounded in $L^2$ by definition and the last inequality follows from Proposition \ref{prop:error-estimate}.\\

$\bullet$ Because of the symmetry of $\mathbb{P}$, by integrating by parts the third integral reads
\begin{align*}
-\frac{\delta}{2} \int_\mathbb{R} w|_{y=0} (\cW_{app}-\cW)^2|_{y=0} \, dx - \frac{\delta}{2} \int_{\mathbb{R}^2_+} (\partial_x u + \partial_y w) (\cW_{app}-\cW)^2 \, dx \, dy,
\end{align*}
which is zero thanks to the zero trace of the velocity field and the divergence free condition.\\

$\bullet$ Integrating by part the term with the Laplacian, one gets 
\begin{align*}
-\varepsilon^6 \nu_0\int_\mathbb{R} \partial_y(u_{app}-u)|_{y=0}(u_{app}-u)|_{y=0} + \partial_y (w_{app}-w)|_{y=0} (w_{app}-w)|_{y=0}  \, dx\\
-\varepsilon^6 \kappa_0 \int_\mathbb{R}\partial_y(b_{app}-b)|_{y=0}(b_{app}-b)|_{y=0} \, dx \\
-\nu_0\varepsilon^6 (\|\nabla(u_{app}-u)\|_{L^2}^2+\|\nabla(w_{app}-w)\|_{L^2}^2)-\kappa_0 \varepsilon^6 \|\nabla(b_{app}-b)\|_{L^2}^2\\
= -\nu_0\varepsilon^6 (\|\nabla(u_{app}-u)\|_{L^2}^2+\|\nabla(w_{app}-w)\|_{L^2}^2)-\kappa_0 \varepsilon^6 \|\nabla(b_{app}-b)\|_{L^2}^2,
\end{align*}
where the last equality follows from the boundary conditions (\ref{BC}).\\

$\bullet $ It remains to deal with the remainder
$$|(R_{app}, \cW_{app}-\cW)_{L^2}| \le \|R_{app}\|^2_{L^2} + \|\cW_{app}-\cW\|_{L^2}^2 \le \delta^2 \varepsilon^4 + \|\cW_{app}-\cW\|_{L^2}^2.$$
The Gronwall inequality yields the result.

\begin{Rmk}
We point out that there is an alternative way to establish the stability inequality, by estimating the last term in the following way:
\begin{align*}
|(R_{app}, \cW_{app}-\cW)_{L^2}| \le \varepsilon^2 \delta^{-1} \|R_{app}\|^2_{L^2} + \delta \varepsilon^{-2} \|\cW_{app}-\cW\|_{L^2}^2 \\
\le  \delta \varepsilon^{-2} \|\cW_{app}-\cW\|_{L^2}^2 + \delta \varepsilon^6.
\end{align*}
This yields 
\begin{equation*}
\|(\cW_{app}-\cW)(t)\|_{L^2(\R^2_+)} \le \delta^\frac{1}{2} \varepsilon^3 \exp((\delta \varepsilon^{-2} t).
\end{equation*}
Note that this last version of the stability estimate allows us to have a stable approximate solution for longer times, of order $\varepsilon^{-1}$, provided that $\delta \le \varepsilon^3$, with a remainder of size $\delta^\frac{1}{2}\varepsilon^3$ ($\delta^\frac{1}{2}\varepsilon^2$ in case of secular growths).\\
However, in the present paper we consider interval of times $[0, t]$ with $t=O(1)$, then we rather rely on the first version of the stability estimate (\ref{inequality:stability}), which provides a smaller remainder for $\delta < \varepsilon^2$.

\end{Rmk}

\section*{Appendix}
We  present a general approach to proving the global existence of weak solutions to system (\ref{original_system})-(\ref{BC}) in $\mathbb{R}^2_+=\mathbb{R} \times \mathbb{R}^+$. We adapt the result of existence of global weak (Leray) solutions  for the incompressible Navier-Stokes equations in \cite{CDGG}, in the case of a general domain of $\mathbb{R}^d, \, d=2,3$. Since there are only minor modifications, we shortly sketch the argument and we refer to the book \cite{CDGG} for a complete proof.\\
We point out that another more explicit approach to prove the global existence existence of (local energy) weak solutions to the 3D incompressible Navier-Stokes equations has recently been developed in \cite{Prange}.\\
We start by recalling the definition of weak solutions to the Cauchy problem associated with system (\ref{original_system})-(\ref{BC}). Some of the notations introduced below have been defined in Section \ref{subsec:stability}.
$$
\ba
&\K:=\left \{ \cW=(u,w,b)\in L^2(\R_+^2)^3,\; \d_x u + \d_y w=0\right\}; \\
&
 \V:=\{ \cW=(u, w, b), \; (u, w) \in H_0^1(\R^2_+)^2, \; b \in H^1(\R^2_+), 	\, \nabla b \cdot n|_{y=0}=0 \};  \\
&\V_\sigma:=\left\{ \cW=(u, w, b) \in \V,  \; \partial_x u + \partial_y w=0 \right\}; \\
&\V'=\text{dual space of } \V; \\
&\V_\sigma'=\text{dual space of } \V_\sigma. \\
& \V_\sigma^\circ=\left\{ \cW=(u, w, b) \in \V', \;  \langle\cW, \cW'\rangle_{\V_\sigma \times \V'}=0 \text{  for any } \cW' \in \V_\sigma \right\}.
\ea
$$
\begin{defi}
A weak solution to the Boussinesq system (\ref{original_system})-(\ref{BC}) on $\R^+ \times \R^2_+$ with initial data $\cW_0 \in \K$ is
$$\cW \in C(\R^+; \V_\sigma') \cap L^\infty_{loc}(\R^+; \K) \cap L^2_{loc}(\R^+; \V_\sigma)$$
such that, for any function $\phi \in C^1(\R^+; \V_\sigma)$,
$$
\ba
\int_{\R^2_+} \cW \cdot \phi  \; dx \, dy & + \int_0^T  \int_{\R^2_+} \mathcal{L} \cW \cdot \phi \; dx \, dy\, dt
 + \varepsilon^6  \int_0^T \int_{\R^2_+}\begin{pmatrix}
 \nu_0 \nabla u\\ \nu_0 \nabla w \\ \kappa_0 \nabla b
\end{pmatrix}  : \nabla \phi \; dx \, dy \, dt \\
& - \int_0^T \int_{\R^2_+} (u, w) \otimes \cW : \nabla \phi \; dx \, dy \, dt -  \int_0^T \int_{\R^2_+} \cW \cdot  \partial_t \phi \; dx \, dy \, dt
\\
& = \int_{\R^2_+} \cW_0 \cdot \phi(t=0) \; dx \, dy.
\ea
$$
Moreover, it satisfies the energy inequality in (\ref{energy-inequality}).
\end{defi}
We now provide a sketch of the proof of Proposition \ref{prop:weak-solutions}. \\
There are essentially two main difficulties in dealing with the half plane $\R^2_+$.\\
$\bullet$ The domain is unbounded: this prevents us from using some compactness results for bounded domains;\\
$\bullet$ The construction of an approximation for the Leray projection (in other words, the construction of an approximation with divergence free velocity field at every step) requires some work to do. In the case of bounded domains, an Hilbertian basis of $\K$ is indeed provided by the spectral theorem for self-adjoint compact operators in Hilbert spaces, since the inverse of the Stokes operator is compact, see \cite{CDGG}. In the whole space, the Leray projector has an explicit expression in terms of the symbol of a pseudodifferential operator homogeneous of degree 0, see \cite{Bertozzi}. In the case of the half space, the embedding of $H_0^1 \subset L^2$ is only continuous and we cannot use the Fourier transform in the vertical direction, then a more implicit approach is needed.\\
The strategy presented in [Chapther 2, \cite{CDGG}] consists in the following main steps. 
\begin{enumerate}
\item Construct a family of orthogonal projectors $\mathbb{P}_\eta, \; \eta \in \R,$ on $\K$. This is done as in [\cite{CDGG}, Chapter 2], by proving that the inverse of the Stokes operator 
$$B: \; \K \rightarrow \V_\sigma \subset \K, \quad B \mathcal{F}=\cW, \quad \cW - \begin{pmatrix}
\nu_0 \varepsilon^6 \Delta u\\
\nu_0 \varepsilon^6 \Delta w\\
\kappa_0 \varepsilon^6 \Delta b\\
\end{pmatrix}-\mathcal{F} \in \V_\sigma^\circ$$
is continuous, self-adjoint, one-to-one, contractive and with range $R(B)$ dense in $\mathbb{K}$. This allows us to define the inverse of $B$ as an unbounded operator $A$ with a dense domain of definition. Since $A$ is self-adjoint, the Stokes operator $A-Id$ is also self-adjoint and then one can apply the spectral theorem to get the following result.
\begin{Prop}
There exists a family of orthogonal projectors on $\mathbb{K}$, denoted by $\mathbb{P}_{\eta}$, with $\eta \in \mathbb{R}$, which commutes with the Stokes operator $A-Id$ and satisfies the following properties.
\begin{itemize}
\item $\mathbb{P}_\eta \mathbb{P}_{\eta'} = \mathbb{P}_{\inf(\eta, \eta')}$ for any $(\eta, \eta') \in \mathbb{R}^2$; 
\item $\mathbb{P}_\eta=0$ for $\eta < 0$ and for any $\cW \in \mathbb{K}$
$$\lim_{\eta \rightarrow \infty} \|\mathbb{P}_\eta \cW - \cW\|_\mathbb{K}=0.$$
\item The family $\mathbb{P}_\eta$ satisfies the right continuity, then for every $\cW \in \mathbb{K}$
$$\lim_{\eta' \rightarrow \eta, \, \eta' > \eta} \|\mathbb{P}_{\eta'} \cW - \mathbb{P}_\eta \cW \|_\mathbb{K}=0.$$
\item For any $\cW \in \mathbb{K}$, the function $\eta \rightarrow (\mathbb{P}_\eta \cW, \cW)_\mathbb{K}=\|\mathbb{P}_\eta \cW\|_\mathbb{K}^2$ is increasing, and
\begin{align*}
&\|\cW\|_{L^2(\mathbb{R}^2_+)}^2=\int_\mathbb{R} d(\mathbb{P}_\eta \cW, \cW)_\mathbb{K}; \\
&\|\nabla \cW\|_{L^2(\mathbb{R}^2_+)}^2=((A-Id)\cW, \cW)_{\mathbb{K}}=\int_\mathbb{R} \eta d(\mathbb{P}_\eta \cW, \cW)_\mathbb{K}.
\end{align*}
\item $\mathbb{P}_\eta \cW \in \cV_\sigma$ for any $\cW \in \mathbb{K}$ and $$\|\nabla \mathbb{P}_\eta \cW\|_{L^2(\mathbb{R}^2_+)} \le \eta^\frac{1}{2} \|\cW\|_{\cV_\sigma}.$$
\end{itemize}
\end{Prop}

\item Show that the family of smoothing operators $\bP_\eta$ can be extended to the dual space $\V'$, see [Corollary 2.2, \cite{CDGG}] and converges to the Leray projector $\bP$ in the sense of the $\K$ norm: for any $\cW \in \K$, $$\lim_{\mu \rightarrow + \infty} \|\bP_\mu \cW - \bP \cW\|_\K=0,$$ as in [Proposition 2.4, \cite{CDGG}].
\item Denoting by $\K_k$ the space $\bP_k \K$ for any integer $k$, define the approximate system of ODEs (Friedrichs approximation)
$$\ba  
\dot{\cW}_k(t)=-\bP_k \cL \cW_k +  \varepsilon^6\bP_k\begin{pmatrix}
\nu_0  \Delta u_k\\
\nu_0  \Delta w_k\\
 \kappa_0 \Delta b_k\\
\end{pmatrix}-\delta \bP_k((u_k \partial_x + w_k \partial_y)\cW_k), \\
\cW_k(0)=\bP_k \cW_0,
\ea$$
where $\cW_k(t) \in \K_k$ for every $k$.
\end{enumerate}
Then the rest of the proof and the uniqueness result follow in the usual way, as for the case of Leray solutions to the 2D incompressible Navier-Stokes equations in bounded domains, see \cite{CDGG} for further details.

\section*{Acknowledgements}

This project has received funding from the European Research Council (ERC) under the European Union's Horizon 2020 research and innovation program Grant agreement No 637653, project BLOC ``Mathematical Study of Boundary Layers in Oceanic Motion''.
This work was  supported by the SingFlows project, grant ANR-18-CE40-0027 of the French National Research Agency (ANR). 
The authors would like to thank Thierry Dauxois for several helpful discussions.

\bibliographystyle{plain}

\end{document}